\documentclass[a4paper]{amsart}

\usepackage{lineno,hyperref,amsmath, amssymb,amsthm}
\usepackage[dvipdfmx]{graphicx}
\usepackage{color}
\usepackage{bm}

\newtheorem{definition}{Definition}[section]
\newtheorem{theorem}{Theorem}[section] 
\newtheorem{lemma}[theorem]{Lemma}
\newtheorem{remark}[theorem]{Remark}
\newtheorem{proposition}[theorem]{Proposition}
\newtheorem{corollary}[theorem]{Corollary}
\newtheorem{example}[theorem]{Example}

\newcommand{\C}{\mathbb{C}}
\newcommand{\R}{\mathbb{R}}
\newcommand{\p}{\partial}

\newcommand{\Hess}{{\rm Hess}}

\newcommand{\fa}{\mathfrak{a}}

\newcommand{\fg}{\mathfrak{g}}

\newcommand{\fk}{\mathfrak{k}}
\newcommand{\fm}{\mathfrak{m}}

\newcommand{\fo}{\mathfrak{o}}

\newcommand{\fp}{\mathfrak{p}}

\newcommand{\fu}{\mathfrak{u}}
\newcommand{\fs}{\mathfrak{s}}

\begin{document}

\title[Nonexistence of stable discrete maps]{Nonexistence of stable discrete maps into some homogeneous spaces of nonnegative curvature}
\dedicatory{Dedicated to Professor Yoshihiro Ohnita on the occasion of his 65th birthday} 

\author[T.Kajigaya]{Toru Kajigaya}
\address{Department of Mathematics, Faculty of Science, Tokyo University of Science
1-3 Kagurazaka Shinjuku-ku, Tokyo 162-8601 JAPAN\endgraf
National Institute of Advanced Industrial Science and Technology (AIST), MathAM-OIL, Sendai 980-8577 
JAPAN
}



\email{kajigaya@rs.tus.ac.jp}


\date{\today}

\begin{abstract}
We consider stabilities for the weighted length or energy functional of a discrete map from a finite weighted graph $(X,m_{E})$  into a smooth Riemannian manifold $(M,g)$.
We prove the non-existence of a stable discrete minimal immersion or a non-constant stable discrete harmonic map from a finite weighted graph into certain homogeneous spaces, such as K\"ahler $C$-spaces of positive holomorphic sectional curvature and some simply-connected compact Riemannian symmetric spaces.
 \end{abstract}

\maketitle

\section{Introduction}

Let $(X,m_E)$ be a finite weighted graph with a positive valued weight function $m_{E}:E\to \R_{>0}$ on the set of (oriented) edges $E$ and $(M,g)$  a smooth closed Riemannian manifold with a Riemannian metric $g$. We consider a piecewise smooth map $f:X\to M$ which realizes a configuration of the graph $X$ in the space $M$, such as a triangulation of surface and a net in a higher dimensional space.   
We define the (weighted) {\it length} ${\bf L}$ and the {\it energy} ${\bf E}$ of $f: (X,m_{E})\to (M,g)$ by
\[
{\bf L}(f):=\sum_{e\in E}m_E(e)\int_0^1\|\dot {f}_e\|_g\,dt,\quad {\bf E}(f):=\frac{1}{2}\sum_{e\in E}m_E(e)\int_0^1\|\dot {f}_e\|_g^2\,dt,
\] 
where $\dot{f}_e$ is the derivative of the restricted map $f_e=f|_e$ for each edge $e\in E$. These functionals are natural generalizations of the length or energy of a closed curve, and a map that appears as a critical point of each functional is regarded as a generalization of the closed geodesic. 

  A critical point of the energy functional ${\bf E}$ is called {\it discrete harmonic map}, and it was introduced by Colin de Verdi\`ere \cite{CdV}  to consider a triangulation of surface consisting of geodesic edges. Kotani-Sunada \cite{KS} studied a special type of discrete harmonic map into a flat torus of arbitrary dimension, known as the standard realization, which provides a periodic realization of a graph  in the Euclidean space $\R^{n}$ with large symmetry. Recently, Lam \cite{Lam} pointed out that the standard realization in a 2-dimensional flat torus is actually obtained by a weighted Delaunay decomposition of the flat torus. Moreover, the existence of a hyperbolic version of the standard realization of a graph in a closed Riemann surface of arbitrary genus has been proved by the author and Tanaka \cite{KT} under a mild assumption of the discrete map. 
Also, a geodesic immersion of a graph with uniform weight function that attains a critical point of ${\bf L}$ is referred to as {\it stationary (or minimal) geodesic net} and studied by many researchers. For instance, various existence results of stationary geodesic nets have been established by Nabutovsky-Rotman \cite{NR} and Liokumovich-Staffa \cite{LiS}.  In this paper, more generally, we shall call an immersion\footnote{A map $f:X\to M$ is an {\it immersion} if $\dot f_{e}\neq 0$ along each $e\in E$. This assumption is necessary to derive a useful characterization of critical point of ${\bf L}$ (see Proposition \ref{prop:pharm}). On the other hand, a discrete harmonic map (i.e. a critical point of ${\bf E}$) is not necessarily an immersion.} $f: (X,m_{E})\to (M,g)$ that is a critical point of ${\bf L}$ {\it discrete minimal immersion}. 
Note that a critical point of ${\bf L}$  is not necessarily a critical point of ${\bf E}$ as long as we fix the weight function $m_E$. However, they are closely related  if we permit a change of weight function (see Section \ref{sec:pre}).

 It is known that if the sectional curvature of $(M,g)$ is non-positive, any discrete harmonic map into $M$ minimizes the energy in its homotopy class (see \cite{CdV, KS}). Thus, we obtain a non-trivial energy minimizing discrete map into such $M$ whenever the map is not homotopic to a point.  For instance, this is the case of previous works \cite{CdV, KS, KT, Lam} in flat tori or hyperbolic surfaces. 
However, the situation is different if the sectional curvature of $(M,g)$ is not always non-positive. For example,  the theorem of Synge implies that, if $M$ is an even-dimensional oriented closed Riemannian manifold with positive sectional curvature, then there is no non-trivial length/energy minimizing closed geodesic in $M$, and in particular, $M$ is simply-connected since a minimizing closed geodesic always exists if $M$ is not simply-connected (indeed, each non-trivial homotopy class in $\pi_{1}(M)$ contains such a  geodesic by the classical result).

When the target manifold $M$ is simply-connected, any continuous map $f:X\to M$ from a graph $X$ is homotopic to a point, and therefore, there is no non-trivial global minimizer of the functional ${\bf F}={\bf L}$ or ${\bf E}$ in the homotopy class of $f$. However, it is still an interesting problem to investigate whether a critical point attains a local minimum of ${\bf F}$ or not.  
 We say a critical point $f$  {\it stable} if $d^2{\bf F}(f_s)/ds^2|_{s=0}\geq 0$ for any piecewise smooth deformation $\{f_s\}_s$ of $f=f_0$. Obviously, any local minimizer must be a stable map.  
 
 The proof of Synge's theorem actually shows that  there is no stable closed geodesic if $M$ is even-dimensional, oriented and has  positive sectional curvature. However, this result is not necessarily true for an odd-dimensional space. Indeed, some concrete examples of stable closed geodesic are given  by Ziller \cite{Zill1}. For example, he showed that the closed geodesic in the direction of the Reeb vector field in a $3$-dimensional Berger sphere $S^{3}$ is stable  (see \cite[Example 1]{Zill1}). Moreover,  a generalization of his example in a higher dimensional Berger sphere $S^{2n+1}$ is given by Torralbo-Urbano (see \cite[Proposition 6]{TU}). Note that the Berger sphere is a Riemannian homogeneous space of strictly positive sectional curvature. Thus, this provides a simple example of stable (discrete) minimal immersion or harmonic map into a simply-connected positively curved homogeneous space. 
 
Furthermore, Cheng \cite{Ch} recently proved that there exists a convex hypersurface $M^{n}$ in $\R^{n+1}$ $(n\geq 2)$ that contains a stable stationary geodesic net consisting of a bouquet graph without any closed geodesic.  He also states that $M$ is homeomorphic to $S^{n}$ and the metric has positive sectional curvature. Note that, by Proposition \ref{prop:key} in the present paper, this also implies the existence of stable discrete harmonic immersion from a weighted graph into such $M$.  In particular, Cheng's result shows that the non-existence of stable discrete maps does not hold in general even if $M$ satisfies the same assumption of Synge's theorem.
 
Therefore,  it is a natural question to ask whether a stable discrete minimal immersion or a stable discrete harmonic map exists in a given simply-connected Riemannian manifold $M$ of positive or non-negative curvature. 
In the present paper, we prove the non-existence of stable maps for certain compact Riemannian homogeneous spaces with an extra geometric structure.  
Our first main result is as follows.

\begin{theorem}\label{thm:main1}
Let $(M,J,g)$ be a simply-connected compact homogeneous K\"ahler manifold of positive holomorphic sectional curvature.   Then there do not exist any
 stable discrete minimal  immersion or a non-constant stable discrete harmonic map from a finite weighted graph $(X, m_{E})$ into $(M,g)$. 
\end{theorem}

A typical example of K\"ahler manifold satisfying the assumption of Theorem \ref{thm:main1} is given by a Hermitian symmetric space of compact type. Moreover, a simply-connected compact homogeneous K\"ahler manifold is also known as a {\it K\"ahler $C$-space} (or a {\it generalized flag manifold}), and it is conjectured that any K\"ahler $C$-space has positive holomorphic sectional curvature (cf. \cite{NZ}). In fact, besides Hermitian symmetric spaces, this conjecture is confirmed by Itoh \cite{Itoh} and Lohove \cite{Loh} for some K\"ahler $C$-spaces with small second Betti number.

Next, we consider a simply-connected compact Riemannian symmetric space (RSS for short) as a target manifold $(M,g)$. Note that the sectional curvature of a compact RSS is non-negative.  In the present paper, we show the following result.

\begin{theorem}\label{thm:main2}
Let $(M,g)$ be a simply-connected, compact irreducible Riemannian symmetric space.
\begin{enumerate}
\item If $(M,g)$ is either a strongly unstable RSS, a Hermitian symmetric space, or satisfies that ${\rm rank}M\leq 3$ and $M\neq SO(q+3)/SO(q)\times SO(3)$ of $q\geq 4$, then there is no stable  discrete minimal immersion from a finite weighted graph $(X, m_{E})$ into $(M,g)$. 
\item If $(M,g)$ is either a strongly unstable RSS, a Hermitian symmetric space or the exceptional Lie group $G_{2}$, then there is no non-contant stable discrete harmonic map from a finite weighted graph $(X, m_{E})$ into $(M,g)$. 
\end{enumerate}
\end{theorem}

Here, a compact RSS $M$ is said to be {\it strongly unstable} if there is no non-constant stable {\it smooth} harmonic map from any compact Riemannian manifold $(N,h)$ into $(M,g)$. Note that there are several equivalent conditions for the strongly instability (see \cite[Theorem 5.3]{HW} and \cite[Theorem 4]{Ohnita}). For example, $M$ is strongly unstable if and only if the identity map $id: (M,g)\to (M,g)$ is not a stable harmonic map.
There is a complete classification of strongly unstable, simply-connected compact irreducible RSSs due to Howord-Wei \cite{HW} and Ohnita \cite{Ohnita} (see Table \ref{list1} in Section \ref{subsec:RSS}).

We conjecture that the non-existence of stable discrete minimal immersion/harmonic map holds for any simply-connected compact RSS. However, there are some difficulties in our method when $M$ is not the space given in Theorem \ref{thm:main2} and these cases are remaining problems. See Section \ref{subsec:RSS} for details and further discussion.

Our proof  is based on the averaging method, which was first used by Lawson-Simons \cite{LS} to prove the non-existence of stable varifolds in the standard sphere. Burns-Burstall-Bartolomeis-Rawnsley \cite{BBBR} and Ohnita-Udagawa \cite{OhU} generalized this method to harmonic maps from a Riemann surface into a compact Hermitian symmetric space, and  if this is the case, they proved that the stable harmonic map must be a holomorphic or anti-holomorphic map. Howard-Wei \cite{HW} and Ohnita \cite{Ohnita} applied the method to a Riemannian manifold isometrically immersed in the Euclidean space, and gave a complete classification of strongly unstable simply-connected compact irreducible RSS. Ohnita \cite{Ohnita3} also used this method to prove the non-existence of stable rectifiable $p$-current of $p$ less than the dimension of the Helgason sphere in some compact RSSs. 

In this paper, we show that their methods are  valid even for  discrete maps.   In fact, Theorem \ref{thm:main1} is proved by applying the method used in \cite{BBBR, OhU} to the energy functional ${\bf E}$. Furthermore, if $M$ is isometrically immersed into the Euclidean space, we obtain some criteria for the non-existence by improving the method given in \cite{HW, Ohnita} (Corollary \ref{cor:cri1} and  Proposition \ref{prop:keyN}).  Theorem \ref{thm:main2} is proved by showing these criteria for an equivariant isometric immersion of symmetric space. However, our proof of Theorem \ref{thm:main2} is divided into several cases and requires a detailed computation for an exceptional case. More precisely,  we construct an explicit isometric immersion of the rank $2$ exceptional compact symmetric space $G_{2}/SO(4)$  in order to prove the non-existence of discrete minimal immersion into $G_{2}/SO(4)$ (See Section \ref{sec:g2} for details).

The paper is organized as follows. In Section \ref{sec:pre}, we summarize basic facts on discrete minimal immersion and discrete ($p$-)harmonic map. In particular, we give a useful characterization and mention a relationship between these maps.  In Section \ref{sec:nonex1}, we give a proof of Theorem \ref{thm:main1}. In Section \ref{sec:nonex2}, we prove criteria for the non-existence when $M$ is isometrically immersed into $\R^{n+k}$. By applying this result to a certain isometric immersion of symmetric space, we prove Theorem \ref{thm:main2}. A detailed computation by using an explicit construction of the isometric immersion  is required when $M=G_{2}/SO(4)$, and we give the details in Section \ref{sec:g2}.

\section{Preliminaries}\label{sec:pre}

Let $X=(V, E)$ be a finite graph, where $V$ is the set of vertices, $E$ is the set of oriented edges. We allow a loop and multiple edges in $X$. We denote the origin and the terminate of an edge $e$ by $o(e)$ and $t(e)$, respectively. The reversed edge of $e$ is denoted by $\overline{e}$. A graph $X$ is said to be {\it connected} if arbitrary two vertices $x,y\in V$ can be joined by a path $c=(e_{1},\ldots, e_{N})$ satisfying that $o(e_{1})=x$, $t(e_{i})=o(e_{i+1})$ for any $i=1,\ldots, N-1$ and $t(e_{N})=y$. Throughout this paper, we assume $X$ is a connected graph. A {\it weight function} $m_E:E\to \R_{>0}$ is a positive valued function on $E$  satisfying that $m_E(\overline{e})=m_E(e)$ for any edge $e\in E$. We call the pair $(X, m_E)$ {\it weighted graph}. 

 In the following, we identify $X$ with a CW-complex as described in \cite{KS}. In particular, each edge $e\in E$ is identified with the closed interval $[0,1]$ and we denote the parameter of $[0,1]$ by $t$. Note that the reversed edge $\overline{e}$ is parameterized by $1-t$.
For a given smooth manifold $M$, we consider a piecewise $C^k$-map $f: X\to M$, which we simply say a {\it discrete ($C^k$-)map}. Here, {\it piecewise $C^k$} means that the restricted map $f_e:=f|_e:[0,1]\to M$ is a $C^k$-map for each $e\in E$.   Note that $f_{\overline{e}}(t)=f_{e}(1-t)$ in this notation.
In the present paper, we assume $k\geq 2$ for our purpose.  

For a given discrete map $f: X\to M$, an edge $e\in E$ is said to be {\it degenerate} if $f_{e}$ is a constant map, otherwise we say {\it non-degenerate}.
A discrete map $f$ is said to be a (discrete) {\it immersion} if $f_e:[0,1]\to M$ is an immersion for each edge $e\in E$, i.e. $\dot{f}_e:=df_{e}/dt\neq 0$ along each $e\in E$. If $f$ is an immersion, then any edge is non-degenerate.

We denote the set of piecewise smooth vector filed along $f$ by $\Gamma(f^{-1}TM)$. More precisely, we set 
$
\Gamma(f^{-1}TM):=
\left\{
W:X\to TM\;\left|\;
  \begin{gathered}
    W_{e}\in \Gamma(f^{-1}_{e}TM)\ \forall e\in E
  \end{gathered}
\right.
\right\}
$,
where $W_{e}:=W|_{e}$. 

\subsection{Discrete minimal immersion and ($p$-)harmonic map}\label{subsec:pharm}
Although our main interest in the present paper is the length or energy functional, we introduce a general notion in order to deal with both functionals simultaneously. 

Let $(X,m_{E})$ be a finite weighted graph and $(M,g)$ a smooth Riemannian manifold.
For an integer $p\geq 1$, we define the {\it $p$-energy} ${\bf E}^p$  of the discrete map $f:(X,m_{E})\to (M,g)$  by 
\[
{\bf E}^p(f):=\frac{1}{p}\sum_{e\in E} m_E(e) \int_0^1 \|\dot{f}_e\|_g^p\, dt,
\]
where $\|\dot{f}_{e}\|_{g}=g(\dot{f}_{e},\dot{f}_{e})^{1/2}$. Note that ${\bf E}^p$ is regarded as a discrete analogue of the $p$-energy between smooth Riemannian manifolds (cf. \cite{BG}). When $p=1$, ${\bf E}^1={\bf L}$ and it coincides with the (twice of) sum of weighted length of $f$. When $p=2$, ${\bf E}^{2}={\bf E}$ and it coincides with the discrete energy considered in \cite{CdV, KS, KT}.

We make a technical, but a reasonable assumption only in the case when $p=1$. Namely, {\it we assume $f$ is an immersion if we consider the critical point of ${\bf E}^1={\bf L}$}. Under this assumption, we obtain the following useful characterization of the critical point of ${\bf E}^{p}$.

\begin{proposition}\label{prop:pharm}
Let $(M,g)$ be a smooth Riemannian manifold, $(X,m_E)$ a  finite  weighted graph and $f:(X,m_{E})\to (M,g)$  a discrete $C^{2}$-map. We set  $T_{e}:=\dot{f}_e$ for each $e\in E$.
\begin{itemize}
\item[(1)] Suppose $p=1$ and $f$ is an immersion. Then, $f$ is a critical point of ${\bf E}^{1}={\bf L}$ if and only if the following two conditions hold:
\begin{enumerate}
\item $f_{e}:[0,1]\to (M,g)$ is a minimal immersion for each edge $e$, i.e. $(\nabla_{T_e}T_e)^{\perp_e}=0$ for the Levi-Civita connection $\nabla$ of $(M,g)$, where $\perp_e$ is the orthogonal projection onto the normal space of $f_e$. This means that the image of $f_e$ is contained in a geodesic segment.
\item For each vertex $x$, it holds that
\[
\sum_{e\in E_x}m_E(e)\|T_e\|^{-1}{T_e}(x)=0,
\]
\end{enumerate}
where $E_x$ denotes the set of edges starting from $x$.
\item[(2)] Suppose $p\geq 2$. Then,  $f$ is a critical point of ${\bf E}^{p}$ if and only if the following two conditions hold:

\begin{enumerate}
\item $f_{e}:[0,1]\to (M,g)$ is a geodesic map for each edge $e$, i.e. $\nabla_{T_e}T_e=0$. 
\item For each vertex $x$, it holds that
\[
\sum_{e\in E_x}m_E(e)\|T_e\|^{p-2}T_e(x)=0.
\]
\end{enumerate}
\end{itemize}
We say the condition {\rm (ii)} for each case the {\rm ($p$-)balanced condition}.
\end{proposition}

\begin{proof}
 Let $\{f_s\}_{s\in (-\epsilon,\epsilon)}$ be a piecewise smooth deformation of a discrete map $f=f_0$.  We denote the variational vector field by $V_{s}:=\p f_{s}/\p s$. Notice that $[V_{e},T_{e}]=0$ for any $e$. 
Then,
\begin{align*}
\frac{d}{ds}{\bf E}^p(f_s)&=\frac{1}{p}\sum_{e\in E}m_E(e)\int_0^1\frac{\partial}{\partial s}\left(g(T_e,T_e)^{\frac{p}{2}}\right)dt=\sum_{e\in E}m_E(e)\int_0^1\|T_e\|^{p-2}g(T_e,\nabla_{V_e}T_e)\,dt\\
&=\sum_{e\in E}m_E(e)\int_0^1g(\|T_e\|^{p-2}T_e, \nabla_{T_e}V_e)dt.
\end{align*}
Note that $\|T_e\|^{p-2}T_e$ is a differentiable vector field along $e$ for any $p\geq 1$ (If $p=1$, this is guaranteed by the assumption that $f$ is an immersion). Thus, we see
\begin{align*}
\frac{d}{ds}{\bf E}^p(f_s)&=\sum_{e\in E}m_E(e)\int_0^1\frac{\p}{\p t}g(\|T_e\|^{p-2}T_e, V_e)-g(\nabla_{T_e}(\|T_e\|^{p-2}T_e), V_e)\,dt\\
&=\sum_{e\in E}m_E(e)\Big[g(\|T_e\|^{p-2}T_e, V_e)\Big]_{t=0}^{t=1}-\sum_{e\in E}m_E(e)\int_0^1g(\nabla_{T_e}(\|T_e\|^{p-2}T_e), V_e)\,dt.
\end{align*} 
Since
$V_e(s,1)=V_{\overline{e}}(s,0)$, $T_e(s,1)=-T_{\overline{e}}(s,0)$
and $m_E(e)=m_E(\overline{e})$, we have
\begin{equation*}
\begin{split}
\sum_{e\in E}m_E(e)g(\|T_e\|^{p-2}T_e, V_e)(s,1)
&=-\sum_{e\in E}m_E(\overline{e})g(\|T_{\overline{e}}\|^{p-2}T_{\overline{e}},V_{\overline{e}})(s,0).
\end{split}
\end{equation*}
Therefore, the first variation becomes
\begin{align}\label{eq:first}
\frac{d}{ds}{\bf E}^p(f_s)=-2\sum_{x\in V}\sum_{e\in E_x}m_E(e)g(\|T_e\|^{p-2}T_e,V_e)_x-\sum_{e\in E}m_E(e)\int_0^1g(\nabla_{T_e}(\|T_e\|^{p-2}T_e), V_e)\,dt.
\end{align}
Since we may assume $V\in \Gamma(f^{-1}TM)$ is arbitrary and $V_e(x)$ is independent of $e\in E_x$,  \eqref{eq:first} shows that $f=f_0$ is a critical point of ${\bf E}^{p}$ if and only if 
\begin{itemize}
\item[${\rm (i)}'$] $\nabla_{T_e}(\|T_e\|^{p-2}T_e)=0$ for any edge $e$, and 
\item[(ii)] $\sum_{e\in E_x}m_E(e)\|T_e\|^{p-2}T_e(x)=0$ for any vertex $x$.
\end{itemize}
Thus, the rest is to show that ${\rm (i)'}$ is equivalent to ${\rm (i)}$ for each $p$ (If $p=2$, this is obvious). On any neighborhood satisfying that $\|T_e\|\neq 0$, we have
\begin{align}\label{eq:i}
\nabla_{T_e}(\|T_e\|^{p-2}T_e)&=(\p_t\|T_e\|^{p-2})T_e+\|T_e\|^{p-2}\nabla_{T_e}T_e\\
&=(\p_t\|T_e\|^{p-2})T_e+\|T_e\|^{p-2}\{(\nabla_{T_e}T_e)^{\top_e}+(\nabla_{T_e}T_e)^{\perp_e}\}\nonumber\\
&=\frac{p-2}{2}\|T_e\|^{p-4}(\p_t\|T_e\|^2)T_e+\frac{1}{2}\|T_e\|^{p-4}(\p_t\|T_e\|^2)T_e+\|T_e\|^{p-2}(\nabla_{T_e}T_e)^{\perp_e}\nonumber\\
&=\frac{p-1}{2}\|T_e\|^{p-4}(\p_t\|T_e\|^2)T_e+\|T_e\|^{p-2}(\nabla_{T_e}T_e)^{\perp_e}\nonumber
\end{align}
 where $\top_e$ (resp. $\perp_e$) is the orthogonal projection onto the tangent (resp. normal) space of $f_e$.  
 
 If $p=1$, \eqref{eq:i} shows that ${\rm (i)}'$ is equivalent to (i) $(\nabla_{T_e}T_e)^{\perp_e}=0$ since $\|T_e\|\neq 0$ along each $e$. Suppose $p\geq 2$. In this case, \eqref{eq:i} implies that  if $\|T_e\|\neq 0$, then 
\[
\nabla_{T_e}(\|T_e\|^{p-2}T_e)=0\ \Longleftrightarrow\ \p_t\|T_e\|^2=0\ {\rm and}\  (\nabla_{T_e}T_e)^{\perp_e}=0\ \Longleftrightarrow\ \nabla_{T_e}T_e=0.
\]
Therefore, ${\rm (i)}'$ is equivalent to (i) on the open subset $\{t\in [0,1]\mid \|T_e(t)\|\neq 0\}$. Since $\nabla_{T_e}T_e(t)=0$ at the point $t$ satisfying that $T_e(t)=0$,  the equivalence of ${\rm (i)}'$ and (i) holds along $e$.  
\end{proof}

 When $p=1$,  the restricted map $f_e:[0,1]\to (M,g)$ is not necessarily a geodesic (i.e., the parameter $t$ is not necessarily proportional to the arc-length parameter), although a critical point of ${\bf E}^{p}$ ($p\geq 2$) is automatically a geodesic map.
With this in mind, we put the following definition:

\begin{definition}
Let $(X, m_{E})$ be a finite weighted graph, and $(M,g)$ a smooth Riemannian manifold.
\begin{itemize}
\item An  immersion $f: (X,m_{E})\to (M,g)$ is said to be a  {\rm (discrete) minimal immersion} if $f$ is a critical point of ${\bf E}^{1}={\bf L}$, or equivalently, $f$ satisfies the conditions {\rm (i)} and {\rm (ii)} in Proposition \ref{prop:pharm} {\rm (1)}. Moreover, a  minimal immersion $f$ is said to be a {\rm minimal geodesic immersion} if $f_{e}:[0,1]\to M$ is a geodesic for each $e\in E$.
\item For $p\geq 2$, a discrete $C^{2}$-map $f: (X,m_{E})\to (M,g)$ is said to be a  {\rm (discrete)  $p$-harmonic map} if $f$ is a critical point of ${\bf E}^{p}$, or equivalently, $f$ satisfies the conditions {\rm (i)} and {\rm (ii)} in Proposition \ref{prop:pharm} {\rm (2)}. If furthermore, $f$ is an immersion, we say $f$ a {\rm $p$-harmonic immersion}. 
\end{itemize}
\end{definition}

\begin{remark}
{\rm The minimal geodesic immersion with uniform weight function $m_{E}\equiv 1$ is also referred to as a {\it stationary geodesic net} or a {\it minimal geodesic net} (cf. \cite{Ch}). Note also that the notion of $2$-harmonic map is the same as the discrete harmonic map considered in \cite{CdV, KS, KT}}.
\end{remark}

Notice that any discrete minimal immersion $f$ becomes a discrete minimal geodesic immersion $f_{geo}$ by changing the parameter. More precisely, for each $e\in E$, we define a bijective differentiable function $s_{e}:[0,1]\to [0,1]$ by
\[
s_{e}(t):=\frac{1}{l_{e}}\int_{0}^{t} \|\dot{f}_{e}\|\,dt,
\]
where $l_{e}$ is the length of $f_{e}$. Then we define $f_{geo}: (X,m_{E})\to (M,g)$ by $f_{geo,e}(t):=f_{e}(s_{e}^{-1}(t))$, where $s_{e}^{-1}$ is the inverse function of $s_{e}$. It is easy to see that $f_{geo}$ is a minimal geodesic immersion satisfying that ${\bf E}^{1}(f)={\bf E}^{1}(f_{geo})$.

 Moreover, 
Proposition \ref{prop:pharm} implies the following simple relation between minimal geodesic immersions and $p$-harmonic immersions:  Let $f: (X,m_E)\to (M,g)$ be a minimal geodesic immersion.  We define a discrete map $f^{p}:(X,m_{E}^{p})\to (M,g)$ by
\begin{align}\label{eq:weight}
&\textup{ $f^{p}:=f$\quad and\quad  $m_E^{p}(e):=m_E(e)\|T_e\|^{1-p}$ for $e\in E$}.
\end{align}
Then, $m_E^{p}:E\to \R$ is a  positive weight  function since $\|T_{e}\|$ is a non-zero constant along each $e\in E$. Note that the definition of $m_{E}^{p}$ depends on $f$. By definition, $f^{p}: (X, m_{E}^{p})\to (M,g)$ is a geodesic map satisfying the $p$-balanced condition. Therefore, $f^{p}$ is a $p$-harmonic immersion by Proposition \ref{prop:pharm}. 
Conversely, if $f$ is a $p$-harmonic immersion, we obtain a minimal geodesic immersion by the reverse procedure. We summarize the result as follows.
\begin{proposition}\label{prop:corr}
Any minimal geodesic immersion $f: (X, m_E)\to (M,g)$ yields a $p$-harmonic immersion $f^{p}: (X, m_E^{p})\to (M,g)$ for any $p\geq 2$ by changing the weight function, and the converse construction also holds. 
\end{proposition}

\begin{example}\label{eg:loop}
{\rm 
 Any closed geodesic in a Riemannian manifold $M$ is regarded as a discrete minimal or $p$-harmonic immersion of a loop $X=(V,E)=(\{\gamma(0)\}, \{\gamma,\overline{\gamma}\})$ with arbitrary weight  $m_{E}$ since the balanced condition is automatically satisfied at the vertex. 
Similarly, if there exist finitely many closed geodesics $\gamma_{1},\ldots, \gamma_{N}$ starting from the same point $w\in M$, then the union of geodesics with arbitrary weight function $m_{E}$ can be regarded as a discrete minimal or a $p$-harmonic immersion of an {\it $N$-bouquet graph} with a single vertex. 

More generally, let $\gamma_{1},\ldots, \gamma_{N}$ $(N\geq 2)$ be simple closed geodesics in $M$ such that any two distinct geodesics intersect transversally, i.e. for any two distinct numbers $i, j\in \{1,\ldots, N\}$, $\gamma_i([0,1])\cap\gamma_j([0,1])$ is either an empty set or consists of finite number of points.  We suppose that, for any $i$,  there exists $j\neq i$ such that $\gamma_i([0,1])\cap\gamma_j([0,1])\neq \phi$ so that there is no isolated geodesic. Then we define the set of vertices $V$ by the set of intersection points:
\begin{align*}
V&:=\{w\in M\mid w\in \gamma_i([0,1])\cap\gamma_j([0,1])\ \textup{for some distinct $i,j\in \{1,\ldots, N\}$}\}.
\end{align*}
If there are at least two intersection points on $\gamma_{i}$, then $\gamma_{i}$ is divided into some geodesic segments and we regard each geodesic segment as a (non-oriented) edge. If there is a single intersection point on $\gamma_{i}$,  we regard $\gamma_{i}$ as a  (non-oriented) loop. Then, we obtain a finite graph $X=(V,E)$ and an embedding $f:X\to M$ such that $f_{e}:[0,1]\to M$ is a geodesic map that image coincides with the corresponding geodesic segment for each $e\in E$. If we consider a uniform weight $m_E\equiv 1$, then $f: (X, m_{E})\to (M,g)$ becomes a discrete minimal geodesic immersion. Note that $f$ is not necessarily a $p$-harmonic map.

A typical example of $(M,g)$ admitting such a union of closed geodesics is given by a {\it compact rank one symmetric space} (CROSS for short) (i.e. $M$ is either $S^{n}$, $\R P^{n}$, $\C P^{n}$, $\mathbb{H}P^{n}$ or $\mathbb{O}P^{2}$ with the canonical Riemannian metric as a symmetric space). In fact,  if $M$ is a CROSS, any geodesic is periodic with the same period $l$.

Moreover, we can construct a discrete minimal immersion into a compact Riemannian symmetric space $M$ of higher rank by a similar method.
Note that not all geodesics are periodic if the rank of $M$ is grater than $1$.  Indeed, the flat torus $T^{n}$ ($n\geq 2$) has infinitely many non-periodic geodesics, and so does $M$ since any compact Riemannian symmetric space contains a totally-geodesic flat torus $T^{r}$ whose dimension $r$ coincides with the rank of $M$ ($T^{r}$ is called the {\it maximal torus} of $M$). However, the flat torus $T^r$ also contains infinitely many periodic geodesics. Since any point $w\in M$ is contained in a maximal torus $T^{r}$, there exist infinitely many direction $v\in T_{w}T^{r}\subseteq T_{w}M$ such that the geodesic in the direction $v$ is periodic. 
Therefore,  one may construct a union of (simple) closed geodesics $\gamma_{1},\ldots, \gamma_{N}$ in $M$ even if $M$ is of rank grater than 1, and this yields a discrete minimal immersion into $M$.
}
\end{example}

See \cite{LiS, NR} for some existence results of discrete minimal immersion in a closed Riemannian manifold. 
Note also that, if $M$ is a compact Riemannian manifold, there exists at least one discrete $2$-harmonic map $f: (X, m_E)\to (M,g)$ (possibly a constant map) in a given homotopy class $\mathcal{C}=[f_0]$ of arbitrary piecewise continuous map $f_0:X\to M$ (See \cite{CdV, KS}). 
See also \cite{KS, KT} for some explicit examples of $2$-harmonic map into a flat torus or a closed Riemann surface.

\subsection{Second variation}
Throughout this subsection, we suppose $f: (X, m_E)\to (M,g)$ is either a minimal immersion with $p=1$ or a $p$-harmonic map with $p\geq 2$.
In this section, we shall compute the second variation of ${\bf E}^{p}$. 

For $V\in \Gamma(f^{-1}TM)$, we denote the orthogonal projection onto the tangent space and normal space of $f_e$ by $V^{\top_e}$ and $V^{\perp_e}$, respectively.

\begin{remark} \label{rem:vf}
{\rm 
For a given $V\in \Gamma(f^{-1}TM)$, $\{V_{e}^{\perp_{e}}\}_{e\in E}$ (and also $\{V_{e}^{\top_{e}}\}_{e\in E}$) does {\it not} define an element in $\Gamma(f^{-1}TM)$ in general (although $V_{e}^{\perp_{e}}\in \Gamma(f_{e}^{-1}TM)$ for each $e\in E$),   because it could admit a vertex $x$ such that $V_{e}^{\perp_{e}}(x)\neq V_{e'}^{\perp_{e'}}(x)$ for some $e,e'\in E_{x}$. 
 }
\end{remark}

\begin{proposition}\label{lem:svf}
Let $f:(X, m_{E})\to (M,g)$ be either a minimal immersion  (not necessarily a geodesic map)  with $p=1$ or a $p$-harmonic map with $p\geq 2$. Then, 
\begin{align}\label{eq:svf0}
\frac{d^2}{ds^2} {\bf E}^p(f_s)\Big{|}_{s=0}=\sum_{e\in E}m_E(e)\int_0^1 \Big{\{}\|\nabla_{T_e}V_e\|^2-g(R(V_e, T_e)T_e, V_e)+(p-2)\|(\nabla_{T_e}V_e)^{\top_e}\|^2\Big{\}}\|T_e\|^{p-2}\,dt,
\end{align}
for any picewise smooth defomation $\{f_{s}\}_{s\in (-\epsilon,\epsilon)}$ of $f=f_{0}$.
\end{proposition}

\begin{proof}
By direct computation, we see
\begin{align} \label{eq:svf1}
\frac{d^2}{d s^2} {\bf E}^p(f_{s})
&=\frac{1}{p}\sum_{e\in E}m_E(e)\int_0^1 \frac{\p^2}{\p s^2}g(T_e, T_e)^{\frac{p}{2}}\, dt=\sum_{e\in E}m_E(e)\int_0^1 \frac{\p}{\p s}g(\nabla_{V_e}T_e, \|T_e\|^{p-2}T_e)\, dt\\
&=\sum_{e\in E}m_E(e)\int_0^1 g(\nabla_{V_e}T_e, \nabla_{V_e}(\|T_e\|^{p-2}T_e))+g(\nabla_{V_e}\nabla_{V_e}T_e, T_e)\|T_e\|^{p-2}\,dt \nonumber\\
&=\sum_{e\in E}m_E(e)\int_0^1 g(\nabla_{T_e}V_e, \nabla_{V_e}(\|T_e\|^{p-2}T_e))+\{-g(R(V_e, T_e)T_e, V_e)+g(\nabla_{T_e}\nabla_{V_e}V_e, T_e)\}\|T_e\|^{p-2}\,dt,
\nonumber
\end{align}
where we used $[V_e, T_e]=0$. On a subset satisfying that $\|T_e\|\neq 0$,  we compute
\begin{align*}
g(\nabla_{T_e}V_e, \nabla_{V_e}(\|T_e\|^{p-2}T_e))&=V_e\|T_e\|^{p-2}\cdot g(\nabla_{T_e}V_e, T_e)+\|T_e\|^{p-2}\|\nabla_{T_e}V_e\|^2\\
&=(p-2)\|T_e\|^{p-4}g(\nabla_{T_e}V_e, T_e)^2+\|T_e\|^{p-2}\|\nabla_{T_e}V_e\|^2\nonumber
\end{align*}
Since $g(\nabla_{T_e}V_e, T_e)^2=\|T_e\|^2\|(\nabla_{T_e}V_e)^{\top_e}\|^2$, we obtain
\begin{align}\label{eq:svf12}
g(\nabla_{T_e}V_e, \nabla_{V_e}(\|T_e\|^{p-2}T_e))&=\{(p-2)\|(\nabla_{T_e}V_e)^{\top_e}\|^2+\|\nabla_{T_e}V_e\|^2\}\cdot \|T_e\|^{p-2}.
\end{align}
Notice that this equation still holds if $\|T_e\|=0$ (Recall that $T_e\neq 0$ if $p=1$).

Next, we shall show  the last term of the formula \eqref{eq:svf1} is vanishing since $f=f_0$ is either a minimal immersion or  a $p$-harmonic map. Indeed, we have
\begin{align*}
&\sum_{e\in E}m_E(e)\int_0^1 g(\nabla_{T_e}\nabla_{V_e}V_e, T_e)\|T_e\|^{p-2}\, dt\Big{|}_{s=0}\\
&=\sum_{e\in E}m_E(e)\int_0^1 \frac{\p}{\p t}g(\nabla_{V_e}V_e, \|T_e\|^{p-2}T_e)-g(\nabla_{V_e}V_e, \nabla_{T_e}(\|T_e\|^{p-2}T_e))\, dt \Big{|}_{s=0}\nonumber\\
&=\sum_{e\in E}m_{E}(e)\Big{[}g(\nabla_{V_e}V_e, \|T_e\|^{p-2}T_e)\Big{]}_{t=0}^{t=1}, \nonumber
\end{align*}
where we used the fact ${\rm (i)}'$ $\nabla_{T_e}(\|T_e\|^{p-2}T_e)=0$ as $f_{0}$ is minimal or $p$-harmonic (See proof of Proposition \ref{prop:pharm}).  Since $\nabla_{V_e}V_e|_{t=1}=\nabla_{V_{\overline{e}}}V_{\overline{e}}|_{t=0}$ and $\nabla_{V_e}V_e(x)$ does not depend on $e\in E_x$, the $p$-balanced condition implies that
\begin{align*}
\sum_{e\in E}m_{E}(e)\Big{[}g(\nabla_{V_e}V_e, \|T_e\|^{p-2}T_e)\Big{]}_{t=0}^{t=1}
&=-2\sum_{x\in V}\sum_{e\in E_{x}}m_E(e)g(\nabla_{V_e}V_e, \|T_e\|^{p-2}T_e)_{x}=0.
\end{align*}

Therefore, substituting \eqref{eq:svf12} to \eqref{eq:svf1}, we obtain the formula.
\end{proof}

According to this proposition, we set
\begin{align}\label{eq:hess}
\Hess_f^p(V,W):=\sum_{e\in E}m_E(e)& \int_0^1 \Big{\{}g(\nabla_{T_e}V_e, \nabla_{T_e}W_e) 
-g( R({V_e}, T_e)T_e, W_e)\\
&\qquad +(p-2)g((\nabla_{T_e}V_e)^{\top_e}, (\nabla_{T_e}W_e)^{\top_e}) \Big{\}}\|T_e\|^{p-2}\,dt\nonumber
\end{align}
for $V, W\in \Gamma(f^{-1}TM)$, and we call the {\it Hessian} of ${\bf E}^{p}$ at $f$.

The following expression of the Hessian is also useful. 
\begin{lemma}\label{lem:hess2}
Let $f:(X, m_{E})\to (M,g)$ be either a minimal immersion  (not necessarily a geodesic map)  with $p=1$ or a $p$-harmonic map with $p\geq 2$.  Then, we have
 \begin{align}\label{eq:hess2}
\Hess_{{f}}^p(V,W)&=\sum_{e\in E}m_E(e)\int_0^1 \Big{\{}g(\nabla_{T_e}V_e^{\perp_{e}},\nabla_{T_e}W_e^{\perp_{e}})
-g(R(V_e^{\perp_{e}}, T_e)T_e, W_e^{\perp_{e}})\\
&\qquad\qquad\qquad\qquad +(p-1)g(\nabla_{T_e}V_e^{\top_{e}},\nabla_{T_e}W_e^{\top_{e}})\Big{\}}\|T_e\|^{p-2}\,dt \nonumber
\end{align}
for $V, W\in \Gamma(f^{-1}TM)$. In particular, if $p=1$, then $\Hess_f^1(V,W)$ depends only on $\{V_e^{\perp_e}\}_{e\in E}$ and  $\{W_e^{\perp_e}\}_{e\in E}$.
\end{lemma}
\begin{proof}
Let $U_{e}$  be any tangent vector field along $f_{e}$ and $N_{e}$  be any normal vector field along $f_{e}$.
Since $(\nabla_{T_{e}}T_{e})^{\perp_{e}}=0$ in any case, we have $g(\nabla_{T_{e}}N_{e},T_{e})=-g(N_{e}, \nabla_{T_{e}}T_{e})=0$, and hence, $\nabla_{T_{e}}N_{e}$ is orthogonal to $T_{e}$. Moreover, we see $g(\nabla_{T_{e}}U_{e}, N_{e})=-g(U_{e},\nabla_{T_{e}}N_{e})=0$ since $U_{e}$ is tangent to $f_{e}$. Therefore $\nabla_{T_{e}}U_{e}$ is also parallel to $T_{e}$. In particular, for any vector field $X_{e}$ along $f_{e}$, the decomposition $\nabla_{T_{e}}X_{e}=\nabla_{T_{e}}X_{e}^{\top_{e}}+\nabla_{T_{e}}X_{e}^{\perp_{e}}$ implies that
\begin{align}\label{eq:svf15}
(\nabla_{T_{e}}X_{e})^{\top_{e}}=\nabla_{T_{e}}X_{e}^{\top_{e}}\quad {\rm and}\quad (\nabla_{T_{e}}X_{e})^{\perp_{e}}=\nabla_{T_{e}}X_{e}^{\perp_{e}}.
\end{align}
Using this relation,  it is easy to see that  \eqref{eq:hess2} follows from  \eqref{eq:hess}. 
\end{proof}

We shall derive another expression of the Hessian for a geodesic map which will be used in the proof of our main theorem.

\begin{lemma}\label{lem:hess3}
Suppose $f$ is either a minimal geodesic immersion with $p=1$ or a $p$-harmonic map with $p\geq 2$. Then, we have
\begin{align}\label{eq:hess3}
&\Hess_f^p(V,W)=\sum_{e\in E} m_E(e)\int_0^1 g(\mathcal{J}_{f,e}^p(V_e), W_e)\,dt+2\sum_{x\in V}\sum_{e\in E_x}m_E(e)g(\mathcal{B}_{f,e}^{p}(V_e),W_e)_x,
\end{align}
where $\mathcal{J}_{f,e}^{p}$ and $\mathcal{B}_{f,e}^{p}$ are defined as follows:
\begin{align*}
\mathcal{J}_{f,e}^p(V_e)&:=
\{-\nabla_{T_e}\nabla_{T_e}V_e-R(V_e, T_e)T_e-(p-2)(\nabla_{T_e}\nabla_{T_e}V_e)^{\top_e}\} \|T_e\|^{p-2}. \\
\mathcal{B}_{f,e}^{p}(V_{e})&:=\{-\nabla_{T_{e}}V_{e}-(p-2)(\nabla_{T_{e}}V_{e})^{\top_{e}}\}\|T_{e}\|^{p-2}.
\end{align*}
\end{lemma}

\begin{proof}

We shall compute the first and the last  integrals in the Hessian formula \eqref{eq:hess}.  Note that $\|T_{e}\|$ is constant along each $e\in E$ since we assume $f$ is a geodesic map.  Then, by using the integration by parts, we have
\begin{align}\label{eq:svf3}
&\sum_{e\in E}m_E(e)\int_0^1 g(\nabla_{T_e}V_e,\nabla_{T_e}W_e)\|T_e\|^{p-2} \,dt\\
&=\sum_{e\in E}m_E(e)\Big{[}g(\nabla_{T_e}V_e, W_e)\cdot \|T_e\|^{p-2}\Big]_{t=0}^{t=1}
-\sum_{e\in E}m_E(e)\int_0^1g(\nabla_{T_e}\nabla_{T_e}V_e,W_e)\cdot \|T_e\|^{p-2} \,dt\nonumber\\
&=-2\sum_{x\in V}\sum_{e\in E_x}m_E(e)g(\nabla_{T_e}V_e,W_e)_x\cdot \|T_e\|^{p-2}
-\sum_{e\in E}m_E(e)\int_0^1g(\nabla_{T_e} \nabla_{T_e}V_e,W_e)\cdot \|T_e\|^{p-2}\,dt.\nonumber
\end{align}

Similar computation shows that 
\begin{align}\label{eq:svf4}
&\sum_{e\in E}m_E(e)\int_0^1 g((\nabla_{T_e}V_e)^{\top_e},(\nabla_{T_e}W_e)^{\top_e})\|T_e\|^{p-2} \,dt\\
&=-2\sum_{x\in V}\sum_{e\in E_x}m_E(e)g((\nabla_{T_e}V_e)^{\top_{e}},W_e)_x\cdot \|T_e\|^{p-2}-\sum_{e\in E}m_E(e)\int_0^1g(\nabla_{T_e}(\nabla_{T_e}V_e)^{\top_e},W_e)\cdot \|T_e\|^{p-2}\,dt.\nonumber
\end{align}
 Here, we have that
\begin{align}\label{eq:svf5}
\nabla_{T_e}(\nabla_{T_e}V_e)^{\top_e}=(\nabla_{T_e}\nabla_{T_e}V_e)^{\top_e}.
\end{align}
Indeed, if $\|T_e\|\equiv 0$ (note that this occurs only when $p\geq 2$), then \eqref{eq:svf5} is obvious. Suppose $\|T_{e}\|\not\equiv 0$, and put $\widetilde{T}_e:=\|T_e\|^{-1}{T_e}$. Then, it holds that 
$
\nabla_{T_e}\widetilde{T}_e=0
$
since $f$ is a geodesic map. Therefore, we see
\begin{align*}
\nabla_{T_e}(\nabla_{T_{e}}V_e)^{\top_e}=\nabla_{T_e}\{g(\nabla_{T_{e}}V_e, \widetilde{T}_e)\widetilde{T}_e\}=g(\nabla_{T_e}\nabla_{T_{e}}V_e, \widetilde{T}_e)\widetilde{T}_e=(\nabla_{T_e}\nabla_{T_{e}}V_e)^{\top_e}
\end{align*}
as required.

Thus, substituting \eqref{eq:svf3}--\eqref{eq:svf5} to the Hessian formula \eqref{eq:hess}, we obtain the formula \eqref{eq:hess3}.
\end{proof}

Notice that, if $p=2$, $\mathcal{J}_{f,e}^{2}$ coincides with the classical Jacobi operator for the map $f_{e}:[0,1]\to M$. It should be noted that, in our discrete setting, the Hessian formula \eqref{eq:hess3} involves the boundary term in general. However, there is a reasonable subset of variational vector fields along $f$ to simplify the Hessian formula \eqref{eq:hess3}:  For a discrete map $f: (X, m_{E})\to (M,g)$, we put
\begin{align}\label{def:balv}
\Gamma_{bal}(f^{-1}TM):=\Big{\{} V\in \Gamma(f^{-1}TM)\mid \sum_{e\in E_x}m_E(e)\nabla_{T_e}V_e(x)=0\,\,\textup{for any vertex $x$}\Big{\}}.
\end{align}
{Geometrically, $\Gamma_{bal}(f^{-1}TM)$ is regarded as the tangent space at $f$ of the space of piecewise $C^{\infty}$ maps satisfying the $2$-balanced condition (cf. \cite[Lemma B.2]{KT}).}  If $V\in \Gamma_{bal}(f^{-1}TM)$, then we have
\begin{align*}
\sum_{x\in V}\sum_{e\in E_x}m_E(e)g(\mathcal{B}_{f,e}^{2}(V_e),W_e)_x=0
\end{align*}
for any $W\in \Gamma(f^{-1}TM)$ since $W_{e}(x)$ does not depend on $e$.  Thus, if this is the case,  the Hessian formula \eqref{eq:hess2} for the $2$-harmonic map can be written by
\[
\Hess_f^2(V,W)=\sum_{e\in E} m_E(e)\int_0^1 g(\mathcal{J}_{f,e}^2(V_e), W_e)\,dt.
\]

\subsection{Stabilities}

A critical point $f$ of ${\bf E}^{p}$ is said to be {\it stable}  if $\frac{d^2}{ds^2} {\bf E}^p(f_s){|}_{s=0}\geq 0$ for any piecewise smooth deformation $\{f_{s}\}_{s\in (-\epsilon, \epsilon)}$ of $f=f_{0}$. If $f$ is either a discrete minimal immersion or a $p$-harmonic map, the stability is equivalent to $\Hess_f^p(V,V)\geq 0$ for any $V\in \Gamma(f^{-1}TM)$.  

The following theorem is an immediate consequence of the second variational formula \eqref{eq:svf0}.

\begin{theorem}\label{thm:sta}
Let $f: (X, m_{E})\to (M,g)$ be either a discrete minimal immersion or a discrete $p$-harmonic map. If the sectional curvature of $(M,g)$ is non-positive, then $f$ is stable.
\end{theorem}

In contrast to this fact, if the sectional curvature of $(M,g)$ is not always non-positive, the stability is a non-trivial property.

We mention some useful facts on the stability of discrete maps. 
Firstly, if $f$ is a discrete minimal immersion, then $f$ is stable if and only if the associated minimal geodesic immersion $f_{geo}$ is stable.  This is because the weighted length ${\bf L}(f)={\bf E}^{1}(f)$ is independent of the choice of parameter of $f_{e}$. 

Secondly, we recall that any minimal geodesic immersion $f:(X,m_E)\to (M,g)$ yields a $p$-harmonic immersion $f^p:(X,m_E^p)\to (M,g)$ by letting $m_E^p(e):=m_E(e)\|T_e\|^{1-p}$ for each $e\in E$ (see Proposition \ref{prop:corr}). Then, we have the following fact.

 \begin{proposition}\label{prop:key}
Let $f: (X, m_{E})\to (M,g)$ be a  minimal geodesic immersion. If $f$ is stable with respect to ${\bf E}^1={\bf L}$, then the associated $p$-harmonic immersion $f^p: (X,m_E^p)\to (M,g)$ is also stable with respect to ${\bf E}^p$.  

In particular, if every non-constant $p$-harmonic map into $(M,g)$ is unstable for some $p\geq 2$, then there is no stable discrete minimal immersion into $(M,g)$.
\end{proposition}

 \begin{proof}
By Lemma \ref{lem:hess2}, we have 
 \begin{align}\label{eq:svf01}
\Hess_{{f}}^p(V,V)=\sum_{e\in E}m_E(e)\int_0^1 \Big{\{}\|\nabla_{T_e}V_e^{\perp_e}\|^2-g(R(V_e^{\perp_{e}}, T_e)T_e, V_e^{\perp_{e}})+(p-1)\|\nabla_{T_e}V_e^{\top_{e}}\|^2\Big{\}}\|T_e\|^{p-2}\,dt,
\end{align}
where $f$ is either a discrete minimal immersion with $p=1$ or a discrete $p$-harmonic map with $p\geq 2$.

Now, we suppose $f:(X,m_{E})\to (M,g)$ is a minimal geodesic immersion and $f^p: (X, m_E^p)\to (M,g)$ is the associated $p$-harmonic immersion. Note that $\|T_{e}\|=\|\dot{f}_{e}\|$ is constant for each $e\in E$ since $f$ is a geodesic map. Moreover, by construction of $f^p$, we have $\|T_{e}\|=\|\dot{f}_{e}\|=\|\dot{f^p_e}\|$ for any $e$ and $\Gamma(f^{-1}TM)=\Gamma((f^{p})^{-1}TM)$. Then, the formula \eqref{eq:svf01} shows that
\begin{align}\label{eq:hessrela}
\Hess_{f^p}^p(V,V)
&=\Hess_{f}^{1}(V,V)+(p-1)\sum_{e\in E}m_{E}(e)\int_{0}^{1}\|\nabla_{T_{e}}{V}_{e}^{\top_{e}}\|^{2}\|T_{e}\|^{-1}\,dt.
\end{align}
  In particular, it holds that 
 \[\Hess_{f^p}^p(V,V)\geq \Hess_{f}^{1}(V,V)\]
  for any  $V\in \Gamma(f^{-1}TM)$, and the equality holds if and only if $\nabla_{T_e}V_e^{\top_e}\equiv 0$ for each $e\in E$. Therefore,  if $f$ is stable then so is $f^p$. 
  \end{proof}
  
  \begin{remark}\label{rem:key}
 {\rm
  More generally, if $f: (X,m_{E})\to (M,g)$ is a $p'$-harmonic immersion, then we obtain a $p$-harmonic immersion $f^{p}: (X,m_{E}^{p})\to (M,g)$ by letting $m_{E}^{p}(e):=m_{E}(e)\|T_{e}\|^{p'-p}$ for any integer $p\geq 2$. Moreover, a similar argument shows that we have  ${\rm Hess}_{f^{p}}^{p}(V,V)\geq {\rm Hess}_{f}^{p'}(V,V)$ if $p\geq p'$. Therefore, the stability of $f$ implies the stability of $f^{p}$ for any $p\geq p'$.
   }
  \end{remark}

 \begin{remark}
 {\rm 
A corresponding argument  for smooth harmonic maps was given in \cite[Lemma 7.1]{OhU2}. More precisely,  it was proved  in \cite{OhU2} that {\it if $F:N\to M$ is a totally geodesic isometric immersion between smooth compact Riemannian manifolds $N, M$,  then $F$ is stable as a harmonic map if and only if $F$ is stable as a minimal immersion and the identity map ${\rm id}_{N}$ is stable}. This fact follows from the following splitting of the Jacobi operator:
 \begin{align}\label{eq:Jacrela}
 J_{F}(V)=J_{F}^{{min}}(V^{\perp})+J_{{\rm id}_{N}}(V^{\top}).
 \end{align}
 where $J_{F}$, $J_{{\rm id_{N}}}$ and $J_{F}^{min}$ are the Jacobi operator of the harmonic map $F$, ${\rm id}_{N}$ and the minimal immersion $F$, respectively, and $V^{\top}$ (resp. $V^{\perp}$) denotes the tangential (resp. normal) vector field of $V$ along $F$. In particular,  for any normal vector field $V=V^{\perp}$ along $F$, we have  $J_{F}^{{min}}(V^{\perp})=J_{F}(V)$, and this implies that,  if $F$ is a stable harmonic map, then $F$ is also stable as a minimal immersion.
  
 We remark that, although the equation \eqref{eq:hessrela} is an analogous formula of \eqref{eq:Jacrela}, we cannot use the same logic for the discrete map $f$ since $\{V_{e}^{\perp_{e}}\}_{e\in E}$ does not define a ``normal vector field along $f$'' (i.e. it possibly takes a multivalue at some vertex, and hence, we do not obtain a well-defined element in $\Gamma(f^{-1}TM)$ in general. See Remark \ref{rem:vf}). It would be an interesting problem to ask whether the converse statement of Proposition \ref{prop:key} holds or not. 
  }
 \end{remark}

\section{Non-existence of stable discrete maps I}\label{sec:nonex1}
In the following, we consider the following properties of a Riemannian manifold $(M,g)$:

\begin{itemize}
\item[$({\bf N}_1)$] {\it There is no stable discrete minimal immersion from a finite weighted graph $(X, m_{E})$ into $(M,g)$.}
\item[$({\bf N}_p)$] {\it There is no non-constant stable discrete $p$-harmonic map from a finite weighted graph $(X, m_{E})$ into $(M,g)$.
}
\end{itemize}

Note that if $(M,g)$ has property $({\bf N}_1)$ (or $({\bf N}_p)$), then $M$ must be simply-connected because there is a stable closed geodesic if $\pi_{1}(M)\neq \{0\}$ by the classical result. Moreover, by Proposition \ref{prop:key} and Remark \ref{rem:key}, it holds that 
$({\bf N}_p)\, \Longrightarrow \, ({\bf N}_{p'})$ for any $1\leq p'\leq p$.

In this section, we apply the averaging method used in \cite{BBBR, OhU} to our discrete maps. As a result, we prove the following non-existence theorem.
 \begin{theorem}\label{thm:unst2}
Let $(M,J,g)$ be a simply-connected compact homogeneous K\"ahler manifold of positive holomorphic sectional curvature. Then $(M,g)$ has properties $({\bf N}_{1})$ and  $({\bf N}_{2})$.
 \end{theorem}

Before starting the proof, we prove a lemma which is easy to see, but will be a key in our proof.
\begin{lemma}\label{lem:key1}
Let $f: (X,m_E)\to (M,g)$ be a discrete map satisfying the $2$-balanced condition (not necessarily a critical point of ${\bf E}^{2}$). If $V$ is a Killing vector field on $M$, then $f^{-1}{V}\in \Gamma_{bal}(f^{-1}TM)$.
 \end{lemma}
\begin{proof}
We fix an arbitrary vertex $x$, and  take any tangent vector $U\in T_xM$.
Since ${V}$ is a Killing vector field on $M$, we have $g(\nabla_X{V},Y)+g(\nabla_Y{V},X)=0$ for any $X, Y\in \Gamma(TM)$. Thus, we see
\begin{align*}
g\Big(\sum_{e\in E_x}m_E(e)\nabla_{T_e}{V}, U\Big)_x&=-\sum_{e\in E_x}m_E(e)g(\nabla_{U}{V}, T_e)_x\\
&=-g\Big(\nabla_{U}{V}, \sum_{e\in E_x} m_E(e)T_e\Big)_x=0
\end{align*}
due to the $2$-balanced condition, where we used the fact that $\nabla_{U}V(x)$ does not depend on $e\in E_x$. Since $U$ is arbitrary,  this shows $f^{-1}{V}\in \Gamma_{bal}(f^{-1}TM)$.
\end{proof}

Next, we recall a basic fact on a simply-connected compact homogeneous K\"ahler manifold $M$. It is well-known that $M$ is isomorphic to an adjoint orbit of a connected compact Lie group $G$, equipped with the canonical complex structure $J$ and a $G$-invariant K\"ahler metric $g$ on the orbit (see  \cite[Theorem 8.89]{Besse}). Thus, we assume $M$ is an adjoint orbit of $G$ in the following.  We will use a property of the canonical complex structure $J$, and thus, we briefly recall the definition of $J$. We refer to \cite[Section 8.B]{Besse} for details. 

We denote the Lie algebra of $G$  by $\fg$ and fix an ${\rm Ad}(G)$-invariant inner product $\langle,\rangle_{\fg}$ on $\fg$. 
Let $M$ be an adjoint orbit in $\fg$, and take an arbitrary point $w\in M\subset \fg$ so that $M={\rm Ad}(G)w=\{{\rm Ad}(g)w\mid g\in G\}$. We always assume $w\neq 0$. We denote the isotropy subgroup of $G$ at $w$ by $K_{w}$ and write its Lie algebra by $\fk_{w}$. Let $\fm_{w}$ be the orthogonal complement of $\fk_{w}$ in $\fg$. Note that $\fm_{w}$ is identified with the tangent space $T_{w}M$ via the correspondence $X\mapsto X^{*}_{w}=\frac{d}{dt}{\rm Ad}({\rm exp}(tX))w|_{t=0}$. 

Let $S_{w}$ be the identity component of the center of $K_{w}$. We denote the Lie algebra of $S_w$ by $\fs_w$. It is easy to see  that $w\in \fs_w$, and hence, we may assume $S_{w}$ is non-trivial.  Then, we obtain an orthogonal decomposition (see \cite[(8.32)]{Besse})
\[
\fm_{w}=\bigoplus_{\alpha\in \Sigma_{+}} \fm_{w,\alpha},\quad {\rm where}\ \fm_{w,\alpha}=\{X\in \fm_{w}\mid {\rm ad}(Z)^{2}X=-\alpha(Z)^{2}X,\ \forall Z\in \fs_{w}\},
\]
where $\alpha$ is a non-zero element in $\fs_{w}^{*}$ satisfying that $\alpha(w)>0$ and $\fm_{w,\alpha}\neq \{0\}$, and we denote the set of such $1$-forms by $\Sigma_{+}$. Using this decomposition, the canonical almost complex structure $J_{w}$ on $\fm_{w}$ is defined by 
\begin{align}\label{eq:Jw}
J_{w}X_{\alpha}=\frac{1}{\alpha(w)}[w,X_{\alpha}]\quad {\rm if}\ X_{\alpha}\in \fm_{w,\alpha}.
\end{align}
Note that each $\fm_{w,\alpha}$ is a $J_{w}$-invariant subspace. 
It then turns out that the canonical almost complex structure $J$ on $M$ is $G$-invariant and integrable (see \cite{Besse} for details).

For each $w\in M$, we extend $J_w$ to $\fg$ by letting $J_wZ=0$ for any $Z\in \fk_w$ so that $J_{w}: \fg\to \fg$ defines an endomorphism on $\fg$.
We use the following fact  which was stated in \cite{BBBR} without proof.
\begin{lemma}\label{lem:key12}
Let $\langle,\rangle_{\fg}$ be an arbitrary ${\rm Ad}(G)$-invariant inner product on $\fg$. Then, for every $w\in M$, $J_{w}$ is skew-symmetric with respect to $\langle, \rangle_{\fg}$.
\end{lemma}
\begin{proof}
Since $\langle, \rangle_{\fg}$ is ${\rm Ad}(G)$-invariant, we have $\langle [w,X],Y\rangle_{\fg}+\langle X,[w,Y]\rangle_{\fg}=0$ for any $X,Y\in \fg$. Thus,  \eqref{eq:Jw} implies that $J_{w}$ is  skew-symmetric w.r.t.  $\langle, \rangle_{\fg}$.
\end{proof}

Let $g$ be a $G$-invariant K\"ahler metric on $M$ compatible with the canonical complex structure $J$. Such a metric always exists, in fact, there exists a unique (up to homothety) K\"ahler-Einstein metric on $M$ (see \cite{Besse}). However, $g$ is not necessarily Einstein in the following. Also, any ${\rm Ad}(G)$-invariant inner product $\langle, \rangle_{\fg}$ defines a $G$-invariant metric $g_{0}$ on $M$ since the ${\rm Ad}(G)$-invariant inner product $\langle, \rangle_{\fg}$ defines a ${\rm Ad}_{G}(K_w)$-invariant inner product on $\fm_w$. Nevertheless, we emphasize that the $G$-invariant metric $g$ is not necessarily $g_0$ in our argument.

Now, we give a proof of Theorem  \ref{thm:unst2}.

\begin{proof}[Proof of Theorem \ref{thm:unst2}]
By Proposition  \ref{prop:key}, it is sufficient to show the instability of any non-constant discrete 2-harmonic map. Thus, we suppose  $f: (X,m_E)\to (M,g)$ is a discrete $2$-harmonic map. For simplicity, we abbreviate the integer $p$ in this proof. For example,  ${\rm Hess}_f$ means ${\rm Hess}_f^2$, etc.

 Let $(M,J,g)$ be a simply-connected homogeneous K\"ahler manifold of positive holomorphic sectional curvature. We identify $M$ with an adjoint of orbit of a compact Lie group $G$ as described above. Since $G$ acts  on $M$ isometrically,  any element $V\in \fg$ gives rise to a Killing vector field $V^{*}$ on $M$ which is defined by 
$
V^*_w:=\frac{d}{dt}{\rm Ad}(\exp(tV)) w|_{t=0} 
$
at each $w\in M$. 
We consider a deformation of $f$ generated by $f^{-1}(JV^*)\in \Gamma(f^{-1}TM)$. We denote $f^{-1}(JV^*)$ simply by $JV^*$, and define a quadratic form $q: \fg\to \R$ by 
\[q(V):=\Hess_f(JV^*,JV^*).\]
 We shall compute the trace of $q$ with respect to an ${\rm Ad}(G)$-invariant inner product $\langle, \rangle_{\fg}$.

 Since $V^{*}$ is a Killing vector field,  we have $f^{-1}V^*\in \Gamma_{bal}(f^{-1}TM)$ by Lemma \ref{lem:key1}, and moreover,  the K\"ahler condition $\nabla J=0$ implies that $f^{-1}(JV^*)\in \Gamma_{bal}(f^{-1}TM)$ (see \eqref{def:balv}). Therefore, by Lemma \ref{lem:hess3}, the Hessian becomes the following simple form:
  \begin{align}\label{eq:keyform}
\Hess_f(JV^{*}, W)=\sum_{e\in E} m_E(e)\int_0^1 g(\mathcal{J}_{f,e}(JV_e^{*}), W_e)\,dt.
\end{align}
for any $W\in \Gamma(f^{-1}TM)$.

It is known that any Killing vector field $X$ satisfies that $(\nabla^2X)(Z,W)+R(X,Z)W=0$ for any $Z,W\in \Gamma(TM)$. Since $f$ is a geodesic map (i.e. $\nabla_{T_{e}}T_{e}=0$) and $\nabla J=0$, we see
\begin{align*}
\nabla_{T_e}\nabla_{T_e}JV^*_e&=J\nabla_{T_e}\nabla_{T_e}V^*_e=J(\nabla^2V^*)(T_e,T_e)=-JR(V^*_e, T_e)T_e.
\end{align*}
Therefore,   we obtain
\begin{align}\label{eq:Jacs}
\mathcal{J}_{f,e}(JV_e^*)=JR(V_e^*, T_e)T_e-R(JV_e^*, T_e)T_e,
\end{align}
and hence, 
\begin{align*}
&q(V)=\sum_{e\in E}m_E(e) \int_0^1 q_w(V)\,dt,\\
 {\rm where}&\  q_w(V):=g(R(V_e^*,T_e)T_e, V_e^*)_w-g(R({JV_e^*},T_e)T_e, JV_e^*)_w
\end{align*}
for $w\in f(e)$.

Take an arbitrary ${\rm Ad}(G)$-invariant inner product $\langle ,\rangle_\fg$ on $\fg$.  We shall compute the trace of $q_w$ for each $w$. Since $J_w$ is skew-symmetric w.r.t.  $\langle ,\rangle_\fg$ by Lemma \ref{lem:key12}, we can take an orthonormal basis of $\fg$ of the form 
\[\{Z_1,\ldots, Z_k,X_1,\ldots, X_n, J_wX_1,\ldots, J_wX_n\},\]
 where $\{Z_1,\ldots, Z_k\}$ is a basis of $\fk_w$ and $\{X_1,\ldots, X_n, J_wX_1,\ldots, J_wX_n\}$ is a basis of $\fm_w$. Note that this choice of basis depends on $w\in M$. By definition, we have $(Z_i)^*_w=0$ for any $i=1,\ldots,k$, and $(J_wX_i)_w^*=J_{w}(X_{i})_{w}^{*}$ for any $i=1,\ldots, n$ under the natural identification $\fm_{w}\simeq T_{w}M$. Therefore, by using this orthonormal basis of $\fg$, it is easy to see that 
\[
{\rm Tr}\,q_w=\sum_{i=1}^{k}q_{w}(Z_{i})+\sum_{i=1}^{n}\{q_{w}(X_{i})+q_{w}(J_{w}X_{i})\}=0.
\]
Therefore,  the trace of $q$ over $\fg$ is also equal to $0$.

Now, we suppose that $f$ is stable. Since ${\rm Tr}\ q=0$, the stability implies that ${\rm Hess}_f(JV^*,JV^*)=0$ for any $V\in \fg$. Moreover,  $\Hess_f$ defines a positive semi-definite inner product on $\Gamma(f^{-1}TM)$, and therefore, the Cauchy-Schwarz inequality shows that 
\[\Hess_f(JV^*,W)^2\leq \Hess_f(JV^*,JV^*)\Hess_f(W,W)=0,\]
and hence, we obtain 
\[\Hess_f(JV^*, W)=0\]
 for any $V\in \fg$ and  $W\in \Gamma(f^{-1}TM)$.  Since ${\rm Hess}_f$ is given by \eqref{eq:keyform}, by taking $W$ whose support is contained in an edge $e\in E$, we obtain $\mathcal{J}_{f,e}(JV_e^*)=0$. By \eqref{eq:Jacs}, this is equivalent to
 \begin{align}\label{eq:key}
JR(V_e^*, T_e)T_e-R(JV_e^*, T_e)T_e=0
 \end{align}
 for each edge $e\in E$ and for any $V\in \fg$.

Fix an arbitrary point $w\in f(e)$. Then, we can take $V$ so that $V^*(w)=T_e(w)$. Then, \eqref{eq:key} shows that 
$
R(JT_e, T_e)T_e(w)=0.
$
Therefore, if $T_e(w)\neq 0$,  this implies the holomorphic sectional curvature of the complex plane spanned by $\{T_e(w), JT_e(w)\}$ is equal to $0$. This contradicts to the assumption that $M$ has positive holomorphic sectional curvature. Thus, it must hold that $T_e(w)=0$ for any $w\in f(e)$ and any $e\in E$, and hence, $f$ is a constant map.
\end{proof}

\begin{remark}
{\rm 
This method is not valid for $p\geq 3$. Indeed, if we put ${q}^{p}(V):={\rm Hess}_{f}^{p}(JV^{*},JV^{*})$ for $V\in \fg$, then we have ${\rm Tr}q^{p}\geq 0$ if $p\geq 3$, and possibly ${\rm Tr}q^{p}>0$.
}
\end{remark}

\section{Non-existence of stable discrete  maps II}\label{sec:nonex2}

In this section, we shall apply another averaging method used in \cite{HW, LS, Ohnita}, and show the non-existence of stable maps in some compact symmetric spaces.

\subsection{Criteria for the non-existence}\label{subsec:cri}

Let $(M^n, g)$ be an $n$-dimensional Riemannian manifold.
Suppose that $(M^n,g)$ is isometrically immersed into the Euclidean space $\R^{n+k}$ by a map $F: M^n\to \R^{n+k}$. For each $w\in M$ we denote the orthogonal projection onto the tangent (resp. normal) space of $F$ at $w$ by  $(\top_F)_w: T_w\R^{n+k}\to T_wM$ (resp.  $(\perp_F)_w: T_w\R^{n+k}\to T_w^{\perp} M$).  Then, the second fundamental form $B$ and the shape operator $A$ of $F$ is defined by $B(Y,Z):=(\overline{\nabla}_YZ)^{\perp_F}$ and $A_{\xi}(Z):=-(\overline{\nabla}_Z\xi)^{\top_F}$ for $Y,Z\in T_wM$ and $\xi\in T^{\perp}_wM$. Note that $B$ is the symmetric tensor satisfying that $g(B(Y,Z),\xi)=g(A_\xi(Y),Z)$.

For a non-zero vector $v\in \R^{n+k}$, we take the parallel vector field $v$ on $\R^{n+k}$, that is, a vector field on $\R^{n+k}$ satisfying that $\overline{\nabla}v=0$. Then, we obtain smooth vector fields $v^{\top_F}$ and $v^{\perp_F}$ along $F$. The following properties are useful. 

\begin{lemma}\label{lem:key2}
For any $v\in \R^{n+k}\setminus\{0\}$ and $Z\in T_{w}M$, we have
\begin{align}
\label{eq:v1}
\nabla_{Z}v^{\top_{F}}&=A_{v^{\perp_{F}}}Z,\\
\label{eq:v2}
 \nabla^\perp_Z v^{\perp_F}&=-B(Z,v^{\top_F}),
\end{align}
where $\nabla^\perp$ denote the normal connection of $F$.
\end{lemma}

\begin{proof}
Since $\overline{\nabla}v=0$, we see
$
\nabla_Zv^{\top_{F}}=(\overline{\nabla}_Zv^{\top_{F}})^{\top_{F}}=(\overline{\nabla}_Zv-\overline{\nabla}_Zv^{\perp_F})^{\top_{F}}=A_{v^{\perp_F}}Z.
$
This proves \eqref{eq:v1}. Moreover, by definition of normal connection, we see
\begin{align*}
\nabla^{\perp}_{Z}v^{\perp_{F}}&=\overline{\nabla}_{Z}v^{\perp_{F}}+A_{v^{\perp_{F}}}(Z)\\
&=\overline{\nabla}_{Z}(v-v^{\top_{F}})+\nabla_{Z}v^{\top_{F}}\\
&=-\overline{\nabla}_{Z}v^{\top_{F}}+\nabla_{Z}v^{\top_{F}}=-B(Z, v^{\top_{F}}).
\end{align*}
This proves \eqref{eq:v2}.
\end{proof}
\begin{remark}
{\rm 
By \eqref{eq:v1}, if a discrete map $f:(X,m_{E})\to (M,g)$ satisfies the $2$-balanced condition, then we have
$
(F\circ f)^{-1}v^{\top_{F}}\in \Gamma_{bal}(f^{-1}TM).
$
 Indeed, for any vertex $x$, the balanced condition shows that 
\begin{align*}
\sum_{e\in E}m_{E}(e)\nabla_{T_{e}}v^{\top_{F}}_{e}(x)&=\sum_{e\in E}m_{E}(e)A_{v^{\perp_{F}}}(T_{e})_{x}=A_{v^{\perp_{F}}}\Big(\sum_{e\in E}m_{E}(e)T_{e}(x)\Big)=0
\end{align*}
since $A_{v^{\perp_{F}}}$ depends only on $x$. }
\end{remark}
 
Let $f:(X,m_E)\to (M,g)$ be either a minimal  immersion with $p=1$ or a $p$-harmonic map with $p\geq 2$.
We shall consider the vector field $(F\circ f)^{-1}v^{\top_F}\in \Gamma(f^{-1}TM)$, and will be denoted by the same symbol $v^{\top_F}$ if there is no confusion. Then, we define a quadratic form on $\R^{n+k}$ by
\begin{align}\label{def:quad}
q_F^p(v):={\rm Hess}_f^p(v^{\top_{F}}, v^{\top_{F}}),
\end{align}
and we shall consider the trace of $q_F^p$ over $\R^{n+k}$. The following trace formula is an extension of the results due to Howard-Wei \cite{HW} and Ohnita \cite{Ohnita}.

\begin{theorem}\label{thm:trf}
Let $(M,g)$ be an $n$-dimensional smooth Riemannian manifold and $f: (X,m_E)\to (M^n,g)$ be either a discrete minimal  immersion with $p=1$ or a discrete $p$-harmonic map with $p\geq 2$. Suppose that $(M,g)$ is isometrically immersed into $\R^{n+k}$ by a map $F: M\to \R^{n+k}$. 

We  define a tensor field on $M$ by
\[
Q_F(X,Y):=g(H, B(X,Y))-2{\rm Ric}(X,Y),
\]
where $B$ is the second fundamental form of $F$, $H$ is the mean curvature vector of $F$ and ${\rm Ric}$ is the Ricci tensor of $(M,g)$. Then, we have
\begin{align}\label{eq:trf}
{\rm Tr}\,q_F^p=\sum_{e\in E\setminus E_{deg}}m_E(e)\int_0^1 \{Q_F(T_e,T_e)+(p-2)\|B(\widetilde{T}_e, \widetilde{T}_e)\|^2\|T_e\|^2\}\|T_e\|^{p-2}\,dt,
\end{align}
where $\widetilde{T}_e:=T_e/\|T_e\|$ and $E_{deg}$ denotes the set of degenerate edges.
\end{theorem}

\begin{proof}
  Note that if $p\geq 2$, $f$ could admit a degenerate edge. However, $e\in E_{deg}$ if and only if $f_e$ is a constant map, or equivalently, $T_e\equiv 0$ along $f_e$ since $f$ is a geodesic map if $p\geq 2$. Thus, the degenerate edge does not affect the Hessian, and hence, we assume $e\in E\setminus E_{deg}$ (or equivalently $\|T_e\|\neq 0$) in the following. 

We use the expression \eqref{eq:svf0} of the second variation.
We put
\begin{align*}
&q_{F,w}^p(v):=\|\nabla_{T_e}v^{\top_{F}}_e\|_{w}^2+(p-2)\|(\nabla_{T_e}v_e^{\top_{F}})^{\top_e}\|_{w}^2-g(R(v^{\top_{F}}_e, T_e)T_e, v^{\top_{F}}_e)_{w}.
\end{align*}
for each $w\in f(e)$ so that
\begin{align*}
q_F^p(v)=\sum_{e\in E\setminus E_{deg}}m_E(e)\int_0^1 q_{F,w}^p(v)\|T_e\|_{w}^{p-2}\,dt.
\end{align*}

We fix arbitrary point $w\in f(e)$, and take an orthonormal basis $\{\mu^1,\ldots, \mu^{n},\nu^{n+1},\ldots, \nu^{n+k}\}$ of $T_w\R^{n+k}$  satisfying that $\{\mu^i\}_{i=1}^n$ is a basis of $T_wM$ and  $\{\nu^{\alpha}\}_{\alpha=n+1}^{n+k}$ is a basis of the normal space $T_w^{\perp}M$ in $\R^{n+k}$. We regard $\mu^i$ ($i=1,\ldots, n$) and $\nu^{\alpha}$ ($\alpha=n+1,\ldots, n+k)$ as parallel vector fields on $\R^{n+k}$. Then, by \eqref{eq:v1}, we have $\nabla_{T_e}(\mu^i)_e^{\top_{F}}(w)=A_{(\mu^i)^{\perp_F}}T_e(w)=0$ for each $i=1,\ldots, n$, and hence, we obtain
\[
\sum_{i=1}^{n}q_{F,w}^p(\mu^i)=-\sum_{i=1}^ng(R(\mu^i, T_e)T_e,\mu^i)_w=-{\rm Ric}(T_e,T_e)_w.
\]
On the other hand,  since $\nabla_{T_e}(\nu^\alpha)_e^{\top_{F}}(w)=A_{\nu^{\alpha}}T_e(w)$ by \eqref{eq:v1}, we have
\begin{align*}
\sum_{\alpha=n+1}^{n+k}q_{F,w}^p(\nu^\alpha)&=\sum_{\alpha=n+1}^{n+k}\|A_{\nu^{\alpha}}T_e\|^2_w+(p-2)\|(A_{\nu^{\alpha}}T_e)^{\top_e}\|^2_w.
\end{align*}
Here, the Gauss equation for the isometric immersion $F:M\to \R^{n+k}$ shows that 
\begin{align*}
\sum_{\alpha=n+1}^{n+k}\|A_{\nu^{\alpha}}T_e\|^2_w&=\sum_{\alpha=n+1}^{n+k}\sum_{i=1}^{n}g(A_{\nu^{\alpha}}T_{e}, \mu^{i})_{w}^{2}=\sum_{i=1}^{n}\|B(T_e, \mu^i)\|^2_w\\
&=\sum_{i=1}^n\Big\{g(B(\mu^i, \mu^i), B(T_e,T_e))_w-g(R(T_e,\mu^i)\mu^i, T_e)_w\Big{\}}\\
&=g(H_e, B(T_e,T_e))_w-{\rm Ric}(T_e,T_e)_w.
\end{align*}
Moreover, by letting $\widetilde{T}_e=T_e/\|T_e\|$, we see
\begin{align*}
\sum_{\alpha=n+1}^{n+k}\|(A_{\nu^{\alpha}}T_e)^{\top_e}\|^2_w&=\sum_{\alpha=n+1}^{n+k}g(A_{\nu^{\alpha}}T_e, \widetilde{T}_e)_{w}^{2}=\sum_{\alpha=n+1}^{n+k} g(B(\widetilde{T}_e, \widetilde{T}_e), \nu^{\alpha})_w^{2}\cdot \|T_e\|_w^2=\|B(\widetilde{T}_e, \widetilde{T}_e)\|_w^2\cdot\|T_e\|_w^2.
\end{align*}
Therefore, we obtain
\begin{align*}
{\rm Tr}\, q_{F}^p&=\sum_{i=1}^n q_{F,w}^p(\mu^i)+\sum_{\alpha=n+1}^{n+k} q_{F,w}^p(\nu^\alpha)\\
&=g(H_e, B(T_e,T_e))_w-2{\rm Ric}(T_e,T_e)_w+(p-2)\|B(\widetilde{T}_e, \widetilde{T}_e)\|_w^2\|T_e\|_w^2.
\end{align*}
This proves the formula.
\end{proof}

If the map $f$ satisfies ${\rm Tr}\,q_F^p<0$, then $f$ must be unstable. By the trace formula \eqref{eq:trf}, we obtain the following sufficient condition  to be ${\rm Tr}\,q_F^p<0$ (resp. ${\rm Tr}\,q_F^1<0$) for any discrete $p$-harmonic map (resp. discrete minimal immersion):
 For the isometric immersion $F: M\to \R^{n+k}$, we set
$\mathcal{S}_{F}:=\{\|B(Z,Z)\|_w^2\mid  w\in M,\, Z\in T_wM\, {\rm with}\, \|Z\|=1\}$ and set 
\begin{align*}
\beta_{F}:={\rm inf}\mathcal{S}_{F},\quad \gamma_{F}:={\rm sup}\mathcal{S}_{F}.
\end{align*}
 \begin{corollary}\label{cor:cri1}
Let $F:(M,g)\to \R^{n+k}$ be an isometric immersion. Then,
\begin{enumerate}
\item If $Q_{F}<\beta_{F}g$, then $(M,g)$ has property $({\bf N}_1)$.
\item If $Q_{F}<-(p-2)\gamma_{F}g$, then $(M,g)$ has property $({\bf N}_p)$.
\end{enumerate}

\end{corollary}

For example, if $F=F_S: S^n(1)\to \R^{n+1}$ is the natural embedding of the unit sphere $S^n(1)$ equipped with the standard metric $g_S$, then we easily see that
$Q_{F_S}=-(n-2)g$
since $B(\cdot,\cdot)=-g(\cdot, \cdot)\nu$ for the outer unit normal vector $\nu$ of $S^{n}(1)$ and ${\rm Ric}=(n-1)g$.  On the other hand, we have $\beta_{F_{S}}=\gamma_{F_{S}}=1$. Thus, Corollary \ref{cor:cri1} shows that the standard sphere $(S^n, g_S)$ has property $({\bf N}_p)$ if $n\geq p+1$.  In particular, the standard $n$-sphere of $n\geq 2$ has property $({\bf N}_1)$. Note that when $n=2$, the property $({\bf N}_2)$ also holds since $S^2$ is isometric to $\C P^1$ and $\C P^1$ has property  $({\bf N}_2)$ by Theorem \ref{thm:unst2}. We remark that this latter fact gives a negative answer for the question posed by  Colin de Verdi\`ere \cite[p.203]{CdV}.

Next, we shall deal with the case when ${\rm Tr}\,q_{F}^{p}=0$. In this special case, we obtain the following obstructions for stable discrete maps: 
\begin{lemma}\label{lem:cri2}
Suppose $(M,g)$ is isometrically immersed into $\R^{n+k}$ by the immersion $F:M\to \R^{n+k}$.
\begin{enumerate}
\item If $f:(X,m_E)\to (M,g)$ is a stable discrete minimal geodesic immersion with ${\rm Tr}\,q_{F}^{1}=0$, then it holds that
\[
K(N\wedge \widetilde{T}_e)_{w}=\|B(N, \widetilde{T}_e)\|_w^2
\]
for any $w\in f(e)$, $e\in E$ and any unit vector $N\in T_wM$ orthogonal to $T_e$, where $K(X\wedge Y)$ denotes the sectional curvature of the plane spanned by $X, Y$ in  $(M,g)$.
\item If $f:(X,m_E)\to (M,g)$ is a non-constant stable discrete $p$-harmonic map with ${\rm Tr}\,q_{F}^{p}=0$, then $K(N\wedge \widetilde{T}_e)_w=0$ for any $w\in f(e)$, $e\in E\setminus E_{deg}$ and any unit vector $N\in T_wM$ orthogonal to $T_e$. Moreover, $F\circ f:(X,m_E)\to \R^{n+k}$ is also a $p$-harmonic map.
\end{enumerate}
\end{lemma}

\begin{proof}
Suppose ${\rm Tr}\,q^{p}_{F}=0$ and $f$ is stable. Then we have $q_{F}^{p}(v)={\rm Hess}_{f}^p(v^{\top_{F}}, v^{\top_{F}})= 0$ for any $v\in \R^{n+k}$, otherwise there exists $v\in \R^{n+k}$ so that ${\rm Hess}_{f}^p(v^{\top_{F}}, v^{\top_{F}})<0$. Moreover, since ${\rm Hess}_{f}^p$ is positive semi-definite on $\Gamma(f^{-1}TM)$, the Cauchy-Schwarz inequality implies that 
${\rm Hess}_{f}^{p}(v^{\top_{F}},W)=0$ for any $v\in \R^{n+k}$ and $W\in \Gamma(f^{-1}TM)$. 
Thus, the Hessian formula \eqref{eq:hess3} shows that 
\begin{align*}
0={\rm Hess}_{f}^{p}(v^{\top_{F}},W)=\sum_{e\in E} m_E(e)\int_0^1 g(\mathcal{J}_{f,e}^{p}(v^{\top_{F}}_e), W_e)\,dt+2\sum_{x\in V}\sum_{e\in E_x}m_E(e)g(\mathcal{B}_{f,e}^{p}(v^{\top_{F}}_e),W_e)_x.
\end{align*}
Since $W$ is arbitrary, by taking $W$ whose support is contained in $e$, we obtain that $\mathcal{J}_{f,e}^{p}(v^{\top_{F}}_e)=0$ along each $e$. In particular, for any non-degenerate edge $e$, we have
\begin{align}\label{eq:ddv1}
0&=-\nabla_{T_e}\nabla_{T_e}v^{\top_{F}}_e-(p-2)(\nabla_{T_e}\nabla_{T_e}v^{\top_{F}}_e)^{\top_e}-R(v^{\top_{F}}_e, T_e)T_e\\
&=-(\nabla_{T_e}\nabla_{T_e}v^{\top_{F}}_e)^{\perp_{e}}-(p-1)(\nabla_{T_e}\nabla_{T_e}v^{\top_{F}}_e)^{\top_e}-R(v^{\top_{F}}_e, T_e)T_e.\nonumber
\end{align}
Here,  by using \eqref{eq:v1}, \eqref{eq:v2} and the assumption that $f$ is a geodesic map, we see
\begin{align}\label{eq:ddv2}
\nabla_{T_e}\nabla_{T_e}v^{\top_{F}}_e&=\nabla_{T_e}(A_{v^{\perp_{F}}}({T}_{e}))\\
&=(\nabla_{{T}_{e}}A)_{v^{\perp_{F}}}({T}_{e})+A_{v^{\perp_{F}}}(\nabla_{{T}_{e}}{T}_{e})+A_{\nabla^{\perp}_{{T}_{e}}v^{\perp_{F}}}({T}_{e}) \nonumber\\
&=(\nabla_{T_{e}}A)_{v^{\perp_{F}}}(T_{e})-A_{B(T_{e}, v^{\top_{F}})}(T_{e}). \nonumber
\end{align}
Notice that the last equation is tensorial. Thus, we fix an arbitrary point $w\in f(e)$, and for any unit tangent vector $N\in T_{w}M$ orthogonal to $T_{e}$, we take $v\in \R^{n+k}$ so that $v(w)=N$.  Then we have $v^{\top_{F}}(w)=N$ and $v^{\perp_{F}}(w)=0$, and by \eqref{eq:ddv1} and \eqref{eq:ddv2}, we see
\begin{align*}
0&=g\Big(-(\nabla_{T_e}\nabla_{T_e}v^{\top_{F}}_e)^{\perp_{e}}-R(v^{\top_{F}}_e, T_e)T_e, N\Big)_w\\
&=g\Big(A_{B({T}_{e}, N)}({T}_{e})-R(N, {T}_e){T}_e,N\Big)_{w}=\{\|B(\widetilde{T}_{e}, N)\|^{2}_{w}-K(N\wedge \widetilde{T}_{e})_{w}\}\cdot \|T_{e}\|^{2}_{w}
\end{align*}
for any unit vector $N\in T_{w}M$ orthogonal to $T_{e}$ and $w\in f(e)$. Therefore, we obtain
\begin{align}\label{eq:ddv3}
K(N\wedge \widetilde{T}_{e})_{w}=\|B(\widetilde{T}_{e}, N)\|^{2}_{w}
\end{align}
if $e\in E\setminus E_{deg}$. Moreover if $p=1$, this proves (i).

Next, we consider the case when $p\geq 2$. In this case, by taking the inner product of \eqref{eq:ddv1} with $T_{e}$, we have
\begin{align*}
0&=(p-1)g(\nabla_{T_e}\nabla_{T_e}v^{\top_{F}}_e, T_{e})_{w}\\
&=(p-1)\{g((\nabla_{T_{e}}A)_{v^{\perp_{F}}}(T_{e}),T_{e})_{w}-g(B(T_{e}, v^{\top_{F}}), B(T_{e},T_{e}))_{w}\}
\end{align*}
by \eqref{eq:ddv2}. Thus, by taking $v\in \R^{n+k}$ so that $v(w)=T_{e}(w)$ for a fixed $w\in f(e)$,  we obtain 
$
\|B(T_{e},T_{e})\|_{w}^{2}=0.
$
Since $w$ is arbitrary, this shows $B(T_{e},T_{e})\equiv 0$ along each $e$.
Because $f_{e}: [0,1]\to (M,g)$ is a geodesic map, this implies that $F\circ f_{e}:[0,1]\to \R^{n+k}$ is also a geodesic map. Since the $p$-balanced condition still holds,  $F\circ f$ is a $p$-harmonic map into $\R^{n+k}$. Moreover, the Gauss equation for the isometric immersion $F:M\to \R^{n+k}$ shows that 
\begin{align}
0=K(N\wedge \widetilde{T}_e)-g(B(\widetilde{T}_e,\widetilde{T}_e), B(N,N))+\|B(\widetilde{T}_e,N)\|^2=2K(N\wedge \widetilde{T}_e),
\end{align}
where we used \eqref{eq:ddv3} and $B(\widetilde{T}_e,\widetilde{T}_e)=0$. Therefore, $K(N\wedge \widetilde{T}_e)=0$ for any $e\in E\setminus E_{deg}$. This proves (ii).
\end{proof}

\subsection{The case of minimal immersion into the sphere}
If $M$ is minimally immersed into the hypersphere $S^{n+k-1}(r)$ in $\R^{n+k}$, the above results are refined as follows:
Let $F:M^{n}\to \R^{n+k}$ be an isometric immersion and suppose $M$ is minimally immersed into $S^{n+k-1}(r)$. We denote the minimal immersion by $\widetilde{F}: M\to S^{n+k-1}(r)$, and write the second fundamental form of $\widetilde{F}$ by $B_{\widetilde{F}}$. Then, $B=B_{F}$ and $B_{\widetilde{F}}$ are related by
\begin{align}\label{eq:brel}
B_{F}(X,Y)=B_{\widetilde{F}}(X,Y)-\frac{1}{r}g(X,Y)\nu
\end{align}
for $X, Y\in\Gamma(TM)$, where $\nu$ is the outer unit normal vector of $S^{n+k-1}(r)$ in $\R^{n+k}$. 
Since $\widetilde{F}$ is a minimal immersion, the relation \eqref{eq:brel} shows that the mean curvature $H$ of $F: M\to \R^{n+k}$ is given by $H=-(n/r)\nu$, and hence, we obtain
\begin{align}\label{eq:qric}
Q_{F}=\frac{n}{r^{2}}g-2{\rm Ric}.
\end{align}
In particular, Corollary \ref{cor:cri1} and Lemma \ref{lem:cri2} implies  the following criteria for the property ${({\bf N}_1)}$ or ${({\bf N}_2)}$.

\begin{proposition}\label{prop:keyN}
Suppose $\widetilde{F}: (M,g)\to S^{n+k-1}(r)$ is a minimal immersion. 
\begin{enumerate}
\item If $(M,g)$ satisfies one of the following two conditions,  then $(M,g)$ has property ${({\bf N}_1)}$.
\begin{enumerate}
\item $\displaystyle  {\rm Ric}>\frac{1}{2}\Big(\frac{n-1}{r^{2}}-\beta_{\widetilde{F}}\Big)g$, where $\beta_{\widetilde{F}}:={\rm inf}\{\|B_{\widetilde{F}}(Z,Z)\|_w^2\mid  w\in M,\, Z\in T_wM\, {\rm with}\, \|Z\|=1\}$.
\item ${\rm Ric}=((n-1)/2r^{2})g$ and for any $w\in M$ and any unit tangent vector $Y\in T_{w}M$, there exists a unit tangent vector $Z\in T_{w}M$ orthogonal to $Y$ such that $K(Y\wedge Z)=0$.
\end{enumerate}
\item If $(M,g)$ satisfies ${\rm Ric}\geq (n/2r^2)g$, then $(M,g)$ has property $({\bf N}_2)$.\end{enumerate}
\end{proposition}
\begin{proof}
(i)  By Corollary \ref{cor:cri1}, $(M,g)$ has property ${({\bf N}_1)}$ if 
\begin{align}\label{eq:q2}
Q_{F}<\beta_{F}g\quad \Longleftrightarrow\quad {\rm Ric}>\frac{1}{2}\Big(\frac{n}{r^{2}}-\beta_{F}\Big)g=\frac{1}{2}\Big(\frac{n-1}{r^{2}}-\beta_{\widetilde{F}}\Big)g.
 \end{align}
 where  we used the fact $\beta_F=\beta_{\widetilde{F}}+(1/r^2)$.
Thus, the condition (a)  provides a sufficient condition for the property ${({\bf N}_1)}$. 

Suppose the condition (b) holds. In this case, we have $Q_{F}=\frac{1}{r^{2}}g$ by \eqref{eq:qric}, and hence, the trace formula \eqref{eq:trf} for the map $F:M\to \R^{n+k}$ and \eqref{eq:brel} show that
\begin{align*}
{\rm Tr}\, q_{F}^1
&=\sum_{e\in E}m_{E}(e)\int_{0}^{1}-\|B_{\widetilde{F}}(\widetilde{T}_{e},\widetilde{T}_{e})\|^{2}\cdot \|T_{e}\|\,dt,
\end{align*}
where $B_{\widetilde{F}}$ is the second fundamental form of $\widetilde{F}:M\to S^{n+k-1}(r)$.
Thus, if the discrete minimal immersion $f: X\to M$ were stable,  it  holds that  
\begin{align}\label{eq:key50}
\textup{${\rm Tr}\, q_{F}^1=0$ and
$B_{\widetilde{F}}(\widetilde{T}_{e},\widetilde{T}_{e})=0$ along each $e\in E$.}
\end{align}
We may assume $f$ is a minimal geodesic immersion. Then, Lemma \ref{lem:cri2} (i) implies that
\begin{align}\label{eq:key5}
K(N\wedge \widetilde{T}_{e})=\|B_{F}(N,\widetilde{T}_{e})\|^{2}=\|B_{\widetilde{F}}(N,\widetilde{T}_{e})\|^{2},
\end{align}
for any unit tangent vector $N\in T_wM$ orthogonal to $T_e(w)$.

By \eqref{eq:key50} and \eqref{eq:key5}, the Gauss equation for $\widetilde{F}:M\to S^{n+k-1}(r)$ becomes
\begin{align*}
K_{S}(N\wedge \widetilde{T}_{e})&=K(N\wedge \widetilde{T}_{e})-g(B_{\widetilde{F}}(\widetilde{T}_{e},\widetilde{T}_{e}), B_{\widetilde{F}}(N,N))+\|B_{\widetilde{F}}(N,\widetilde{T}_{e})\|^{2}\\
&=2K(N\wedge \widetilde{T}_{e}).
\end{align*}
for any unit vector $X\in T_{w}M$ orthogonal to $T_{e}(w)$, where $K_{S}$ denotes the sectional curvature of $S^{n+k-1}(r)$. 
Therefore, we conclude
$
K(N\wedge \widetilde{T}_{e})>0
$
for any unit vector $N\in T_{w}M$ orthogonal to $T_{e}(w)$. However, this contradicts to the assumption of (b). Thus, there is no stable discrete minimal immersion into $M$.

(ii) By Corollary \ref{cor:cri1}, $(M,g)$ has property ${({\bf N}_2)}$ if $Q_F<0$, and this is equivalent to ${\rm Ric}>(n/2r^2)g$. Suppose ${\rm Ric}=(n/2r^2)g$, or equivalently, $Q_F=0$. Then, we have ${\rm Tr} q_F^2=0$ by the trace formula. Thus, by Lemma \ref{lem:cri2} (ii), if a non-constant $2$-harmonic map $f:(X,m_E)\to (M,g)$   were stable, then $F\circ f$ is also a $2$-harmonic map into $\R^{n+k}$, and hence, $F(M)$ contains a line segment in $\R^{n+k}$.  However, this yields a contradiction because $F(M)\subseteq S^{n+k-1}(r)$ and $S^{n+k-1}(r)$ does not contain any line segment.  Thus, there is no non-constant stable discrete 2-harmonic map into $M$.
\end{proof}

\subsection{Isotropy irreducible homogeneous spaces}\label{ssec:eg}
As studied in \cite{HW, Ohnita},  we shall apply our non-existence criteria to an isotropy irreducible compact Riemannian homogeneous space. 

Let $(M,g)$ be a compact Riemannian manifold and suppose a closed Lie subgroup $G\subseteq {\rm Isom}(M,g)$ acts  on $M$ transitively. We denote the isotropy subgroup at a point $w\in M$ by $K$ so that $M$ is diffeomorphic to $G/K$, and we always use this identification in the following. Note that $G$ is a compact Lie group since $M$ is compact. For the sake of convenience, we fix an ${\rm Ad}(G)$-invariant inner product $\langle, \rangle$ on the Lie algebra $\fg$ of $G$. We denote the Lie algebra of $K$ by $\fk$ and write the orthogonal complement of $\fk$ in $\fg$ by $\fm$. Then, we obtain a reductive decomposition $\fg=\fk\oplus\fm$, and $\fm$ is identified with the tangent space $T_o(G/K)$ at the origin $o$.  

We say $(M=G/K,g)$ is {\it isotropy irreducible} if the isotropy representation is irreducible, or equivalently, the restriction ${\rm Ad}|_K: K\to GL(\fm)$ is an irreducible representation. If this is the case, every ${\rm Ad}_{G}(K)$-invariant inner products on $\fm$ is proportional by Schur's lemma, and hence, a $G$-invariant Riemannian metric $g$ is uniquely determined  (up to homothety).  In particular, we may assume $g$ is obtained by the extension of the restriction $\langle, \rangle|_\fm$ of the ${\rm Ad}(G)$-invariant inner product, namely, $g$ is a normal homogeneous metric. Moreover, it turns out that $g$ is actually an Einstein metric (see \cite{Wolf}) and the sectional curvature of $(M,g)$ is non-negative (see \eqref{eq:nsc} below for the precise expression of the curvature).  If $M$ is flat, every discrete harmonic map is stable by Theorem \ref{thm:sta}. Thus we assume $M$ is a non-flat space for our purpose.  See \cite{WZ,Wolf} for the classification of simply-connected isotropy irreducible Riemannian homogeneous spaces.

  Let us recall the $i$-th standard immersion  $\Phi_i: M^n\to \R^{n+k}$ of $M=G/K$. For a positive integer $i\geq 1$, we denote the $i$-th (positive) eigenvalue of the Laplace operator $\Delta=d\delta+\delta d$ acting on $C^{\infty}(M)$ by $\lambda_i>0$. We denote the $i$-th eigenspace in $L^2(M)$ by $V_i$ and set $m_i:={\rm dim}V_i$. Take an $L^2$-orthonormal basis $\{f_1,\ldots, f_{m_i}\}$ of $V_i$ and define an immersion $\Phi_i: M\to V_i\simeq \R^{m_i}$ by
  \[
  \Phi_i:=\alpha\cdot (f_1,\ldots, f_{m_i}),
  \]
  where $\alpha$ is some positive constant.   
It turns out that $\Phi_i$ is a $G$-equivariant map with respect to some orthogonal representation $\rho: G\to O(V_i)$, i.e. it holds that $\Phi_i(gw)=\rho(g)\Phi_i(w)$ for any $g\in G$ and $w\in M$ (see \cite[Subsection 2.4.2]{BCO} for details). Thus, the induced metric is also $G$-invariant, and  proportional to $g$ since $M=G/K$ is isotropy irreducible. Therefore, we can choose $\alpha$ so that $\Phi_i:(M,g)\to \R^{m_i}$ is an isometric immersion. Then, the map $\Phi_i$ is called  the {\it $i$-th standard isometric immersion of $(M,g)$.} In the following, we suppose $\Phi_{i}$ is an isometric immersion.
  
    Because $\Phi_i$ is equivariant under the representation $\rho:G\to O(V_i)$, $\Phi_i$ is an immersion into the hypersphere $S^{m_i-1}(r)$ of some radius $r$. We denote the immersion by $\widetilde{\Phi}_i: M\to S^{m_i-1}(r)$. Then, by the Theorem of Takahashi, $\widetilde{\Phi}_{i}$ becomes a minimal immersion. Indeed, a well-known formula for an isometric immersion into the Euclidean space shows that $\Delta\Phi_i=-H$, where $H$ is the mean curvature vector of $\Phi_i$. Since  the mean curvature vector $\widetilde{H}$ of $\widetilde{\Phi}_i$ and  $H$ is related by $H=\widetilde{H}-\frac{n}{r}\nu$,  we actually have $\Delta\Phi_i=-\widetilde{H}+\frac{n}{r}\nu$. On the other hand,  we have $\Delta\Phi_i=\lambda_i\Phi_i=\lambda_i\cdot r\nu$ by definition of $\Phi_i$. Combining these two equations,  we see that 
    $\widetilde{H}=0$ and $r^2={n/\lambda_i}$,
    namely, $\widetilde{\Phi}_i: M\to S^{m_i-1}(r)$ is a minimal immersion into the hypersphere of radius $r=\sqrt{n/\lambda_i}$. 
    
 Since $g$ is an Einstein metric, we have ${\rm Ric}=c\cdot  g$ for some constant $c$. Note that $c$ is written by $c={\rm scl}(M)/n$, where ${\rm scl}(M)$ is the scalar curvature of $(M,g)$. Since $\widetilde{\Phi}_i$ is a minimal immersion,   \eqref{eq:qric} shows that
 \begin{align*}
 Q_{\Phi_i}=(\lambda_i-2c)g.
 \end{align*}
Moreover, by applying Proposition \ref{prop:keyN} to the minimal immersion $\widetilde{\Phi}_i$, we obtain the following:

\begin{proposition}\label{prop:lb}
Let $M^{n}=G/K$ be a non-flat, isotropy irreducible compact homogeneous space with ${\rm Ric}=c\cdot g$. 
\begin{enumerate}
\item If $(M,g)$ satisfies one of the following two conditions, then $(M,g)$ has property $({\bf N}_1)$.
\begin{enumerate}
\item[${\rm (a)}'$] $\lambda_{i}<\frac{n}{n-1}(2c+\beta_{\widetilde{\Phi}_{i}})$ for some $i$.
\item[${\rm (b)}'$] $\lambda_{1}=2cn/(n-1)$ and ${\rm dim}\fa_{Y}\geq 2$ for any $Y\in \fm$, where $\fa_Y$ is a maximal abelian subspace in $\fm$ containing $Y$.
\end{enumerate}
\item  If $(M,g)$ satisfies $\lambda_1\leq 2c$, then $(M,g)$ has property $({\bf N}_2)$.
\end{enumerate}
\end{proposition}
\begin{proof}
By Proposition \ref{prop:keyN} (i)-(a), the property (${\bf N}_1$) is provided if $\widetilde{\Phi}_i$ satisfies ${\rm Ric}>\frac{1}{2}(\frac{n-1}{r^{2}}-\beta_{\widetilde{\Phi}_{i}})g$ for some $i$.  Since ${\rm Ric}=cg$ and $r=\sqrt{n/\lambda_{i}}$,  this condition is equivalent to ${\rm (a)}'$.

Next, we shall apply Proposition \ref{prop:keyN} (i)-(b) to the isometric immersion $\widetilde{\Phi}_i$.
It is easy to see that ${\rm Ric}=((n-1)/2r^{2})g$ if and only if $\lambda_{i}=2cn/(n-1)$ since $r=\sqrt{n/\lambda_{i}}$ and $g$ is an Einstein metric. If $i\geq 2$, then $\lambda_1<\lambda_i=2cn/(n-1)$ and this is reduced to the case ${\rm (a)}'$. Thus, we suppose $\lambda_{1}=2cn/(n-1)$.

We consider the obstruction for the sectional curvature given in  Proposition \ref{prop:keyN} (i)-(b). As mentioned above, we may assume $g$ is a normal homogeneous metric. 
Thus, by using the general formula of the curvature tensor of Riemannian homogeneous space (see \cite[Theorem 7.30]{Besse}), we obtain that
\begin{align}\label{eq:nsc}
K(Y\wedge Z)=|[Y,Z]_{\fk}|^{2}+\frac{1}{4}|[Y,Z]_{\fm}|^{2}
\end{align}
for any orthonormal pair $Y, Z\in \fm$, where $Y_{\fk}$ and $Z_{\fm}$ denote the orthogonal projection onto $\fk$ and $\fm$, respectively. In particular, we see that $K(Y\wedge Z)=0$ if and only if $[Y,Z]=0$. Therefore, for a given $Y\in \fm$, there exists a unit vector $Z\in \fm$ orthogonal to $X$ such that $K(Y\wedge Z)=0$ if and only if ${\rm dim}\fa_{Y}\geq 2$. Since $M$ is homogeneous, this shows that the condition ${\rm (b)}'$ implies the property $({\bf N}_1)$. This proves (i).

Similarly, by applying Proposition \ref{prop:keyN}-(ii) to $\widetilde{\Phi}_{i}$, we obtain (ii).
 \end{proof}
\begin{remark}\label{rem:lb}
{\rm
It is not necessary to assume  $M$ is simply-connected in this proposition. In particular, if there exists a stable discrete minimal immersion into an isotropy irreducible space $M=G/K$ (e.g. $\pi_{1}(M)\neq \{0\}$), then we have 
\begin{align*}
\lambda_{1}\geq \frac{2cn}{n-1}=\frac{2{\rm scl}(M)}{n-1}.
\end{align*} 
We remark that, it is known $\lambda_{1}\geq {\rm scl}(M)/(n-1)$ for any isotropy irreducible compact homogeneous space (see \cite[Theorem 2]{Nagano}). Thus,  an improved estimate  of $\lambda_{1}$ is provided if $M$ admits a stable discrete minimal immersion.
}
\end{remark}

Notice that if 
\begin{align}\label{eq:cond1}
\lambda_1<2c
 \end{align}
(or equivalently $Q_{\Phi_i}<0$ for some $i$), then $(M,g)$ has properties $({\bf N}_1)$ and $({\bf N}_2)$. 
It is known that this condition \eqref{eq:cond1} is actually equivalent to the {\it strongly instability} of $M=G/K$ in the sense of \cite{HW} (see \cite[Theorem 5.3]{HW}. See also \cite[Theorem 4]{Ohnita}). If $M$ is strongly unstable, there is no non-constant stable {\it smooth} harmonic map from any compact Riemannian manifold $N$ into $M$. Our result shows that a similar non-existence result still holds for {\it discrete} harmonic maps.
 \begin{corollary}\label{cor:su}
If $(M,g)$ is a strongly unstable, isotropy irreducible compact homogeneous space, then $(M,g)$ has properties $({\bf N}_1)$ and $({\bf N}_2)$. \end{corollary}

\subsection{Compact symmetric spaces}\label{subsec:RSS}

A typical example of isotropy irreducible homogeneous space is given by an irreducible Riemannian symmetric space of compact type (RSSCT for short). We consider the canonical Riemannian metric $g_K$ on RSSCT $M=G/K$ which is defined by the Killing form ${\bf B}$ on the semi-simple Lie group $G$. Then, it is known that the Ricci tensor of $(M,g_{K})$ is given by ${\rm Ric}=(1/2)g_{K}$.

The first eigenvalue $\lambda_{1}$ of simply-connected RSSCT with respect to $g_{K}$ is given in \cite[p.187]{KOT}. Using this data, we can determine strongly unstable simply-connected  RSSCTs, which was first done by Howard-Wei and Ohnita \cite{HW, Ohnita} (See Table \ref{list1}).   Notice that there are three classes of simply-connected RSSCT:
\begin{enumerate}
\item $M$ is strongly unstable. This is equivalent to $\lambda_1<1$.
\item $M$ is either a Hermitian symmetric space or $G_{2}$. This is equivalent to $\lambda_1=1$.
\item Otherwise, i.e. $\lambda_1>1$.
\end{enumerate}

 \begin{table}[htb]
  \begin{tabular}{|c|c|c|c|c|}\hline
 $M$  & Rank & SU or HSS  \\ \hline\hline
 $SU(p)/SO(p)$ $(p\geq 3)$  & $p-1$ &  No   \\ \hline
  $SU(2p)/Sp(p)$ $(p\geq 2)$ & $p-1$ & SU  \\ \hline
$SU(q+p)/S(U(q)\times U(p))$ ($q\geq p\geq 1$)& $p$  & HSS\\ \hline
\begin{tabular}{c}
$SO(q+1)/SO(q)$ ($q\geq 3$) \\
$SO(q+2)/SO(q)\times SO(2)$ ($q\geq 3$) \\
$SO(q+p)/SO(q)\times SO(p)$ ($q\geq p\geq 3$) 
\end{tabular}
& 
\begin{tabular}{c}
1\\
2\\
$p$
\end{tabular}
& 
\begin{tabular}{c}
SU\\
HSS\\
No 
\end{tabular}
\\ \hline
 $Sp(p)/U(p)$ $(p\geq 1)$ & $p$ & HSS \\ \hline
$Sp(q+p)/Sp(q)\times Sp(p)$ ($q\geq p\geq 1$)& $p$ & SU  \\ \hline
$SO(2p)/U(p)$ $(p\geq 2)$ & $[p/2]$ & HSS\\ \hline 
$E_{6}/Sp(4)$ & $6$ & No \\\hline
 $E_{6}/SU(6)Sp(1)$ & $4$ & No \\ \hline 
 $E_{6}/Spin(10)U(1)$  & $2$ & HSS\\ \hline 
 $E_{6}/F_{4}$ & $2$ & SU\\ \hline
 $E_{7}/SU(8)$  & $7$ & No\\ \hline
 $E_{7}/SO(12)Sp(1)$  & $4$ & No\\ \hline
 $E_{7}/E_{6}U(1)$  & $3$ & HSS\\ \hline
 $E_{8}/SO(16)$  & $8$ & No\\ \hline
 $E_{8}/E_{7}Sp(1)$  & $4$ & No\\ \hline
 $F_{4}/Sp(3)Sp(1)$  & $4$ & No\\ \hline
 $F_{4}/Spin(9)$  & $1$ & SU\\ \hline
 $G_{2}/SO(4)$ & $2$ & No\\ \hline
  $SU(p+1)$ $(p\geq 1)$ & $p$ &SU \\ \hline 
\begin{tabular}{c}
$Spin(5)$\\
$Spin(2p+1)$ $(p\geq 3)$
\end{tabular}
 &
\begin{tabular}{c}
 $2$\\
  $p$ 
  \end{tabular}
  & 
  \begin{tabular}{c}
  SU\\
  No
  \end{tabular}
   \\ \hline
 $Sp(p)$ $(p\geq3)$ & $p$ & SU\\ \hline 
 $Spin(2p)$ $ (p\geq4)$ & $p$ & No\\ \hline
 $E_{6}$  & $6$ & No\\ \hline
 $E_{7}$  & $7$ & No\\ \hline
  $E_{8}$  & $8$ & No\\ \hline
  $F_{4}$  & $4$ & No\\ \hline
  $G_{2}$  & $2$ & No\\ \hline
  \end{tabular}
  \caption{List of irreducible simply-connected Riemannian symmetric space of compact type. The second column shows the rank of $M$, that is, the dimension of the maximal abelian subspace in $\fm$. The third column indicates whether $M$ is strongly unstable (SU) or Hermitian symmetric (HSS), where ``No'' indicates $M$ is neither of them.}
      \label{list1}
  \end{table}
  
 By Proposition \ref{prop:lb} (ii), we obtain the following result.
  \begin{corollary}
  Let $(M,g_K)$ be a simply-connected compact irreducible Riemannian symmetric space. If $M$ is either a strongly unstable RSS,  a Hermitian symmetric space or $G_2$, then $M$ has property $({\bf N}_2)$ (and also $({\bf N}_1)$).
  \end{corollary}
For example, any simply-connected compact rank one Riemannian symmetric space (CROSS) has properties  $({\bf N}_1)$ and $({\bf N}_2)$. If $M$ is a Hermitian symmetric space, this corollary also follows from Theorem \ref{thm:unst2}. 

For the property $({\bf N}_1)$, we can slightly improve the result.  Since any simply-connected CROSS has property $({\bf N}_1)$, we suppose ${\rm rank}M\geq 2$ in the following. In this case, Proposition \ref{prop:lb} (i) implies the following:
\begin{corollary}\label{cor:lb}
Let $(M^{n},g_{K})$ be an irreducible compact Riemannian symmetric space of rank$M\geq 2$. If the inequality 
\begin{align}\label{eq:cond4}
\lambda_{1}\leq \frac{n}{n-1}
\end{align}
holds, then $(M,g_{K})$ has property $({\bf N}_1)$.
 \end{corollary}

\begin{proof}
 It is well-known that, if $M$ is a symmetric space, the isotropy representation is a polar action and the section is given by a maximal abelian subspace $\fa$ in $\fm$. Namely, any element $Y\in \fm$ is expressed by $Y={\rm Ad}(k)Y'$ for some $k\in K$ and $Y'\in \fa$. Then, $\fa_{Y}:={\rm Ad}(k)\fa$ is a maximal abelian subspace in $\fm$ containing $Y$  since ${\rm dim}\fa_{Y}={\rm dim}\fa$. Because we assume ${\rm rank}M\geq 2$, it automatically satisfies that ${\rm dim}\fa_{Y}\geq 2$ for any $Y\in \fm$.  Thus, Proposition \ref{prop:lb} (i) implies the desired result.
 \end{proof}

By the classification result and the data of the first eigenvalue given in \cite{KOT}, we see that $M$ satisfies the inequality \eqref{eq:cond4} if and only if $M$ is either a strongly unstable RSS, a Hermitian symmetric space or one of the following symmetric spaces:
\begin{align*}
SU(3)/SO(3),\quad SU(4)/SO(4)\simeq SO(6)/SO(3)\times SO(3),\quad Spin(7),\quad G_{2}.
\end{align*}
Note that when $M=SU(4)/SO(4)$ or $Spin(7)$, then the equality holds in \eqref{eq:cond4}. Moreover, the rank of the above spaces are equal to either $2$ or $3$.  If $M$ is one of the above spaces, we conclude that $M$ has property $({\bf N}_1)$ by Corollary \ref{cor:lb}.
However, if $M$ is neither strongly unstable nor Hermitian symmetric and ${\rm rank}M\geq 4$, then \eqref{eq:cond4} does not hold, and hence,  the property $({\bf N}_1)$ does not follow from Corollary \ref{cor:lb} for the higher rank case.

Let us consider the case when {\rm rank}$M\leq 3$ and not satisfying \eqref{eq:cond4}. The remaining spaces are
\begin{align*}
G_{2}/SO(4),\quad SO(q+3)/SO(q)\times SO(3)\ (q\geq 4).
\end{align*}
We shall deal with the case when $M=G_{2}/SO(4)$ in the next section, and by using Corollary \ref{cor:cri1} (i) (or Proposition \ref{prop:keyN} (i)), we show that the property $({\bf N}_1)$ holds even in this case.  Summarizing these results, we obtain our main result Theorem \ref{thm:main2}.

 We mention the case of remaining rank $3$ symmetric space, that is, the $3$-plane oriented real Grassmannian manifold $\widetilde{G}_{3}(\R^{q+3})\simeq SO(q+3)/SO(q)\times SO(3)$ of $q\geq 4$. In order to examine the criteria given in Proposition \ref{prop:keyN} (i),  the minimal immersion $\widetilde{\Phi}_{1}$ associated with the first standard immersion $\Phi_1$ of  $\widetilde{G}_{3}(\R^{q+3})$ would be a reasonable choice. However, our claim here is that {\it the map $\widetilde{\Phi}_1$ does not satisfy the sufficient conditions for the property $({\bf N}_1)$ given in Proposition $\ref{prop:keyN}\ {\rm (i)}$ whenever $q\geq 4$ and $q\neq 6$.} Indeed, the condition given in Proposition \ref{prop:keyN}\ {\rm (i)}-(b) is reduced to  $\lambda_1=2cn/(n-1)=n/(n-1)$ and it does not hold when $(M,g)=(\widetilde{G}_{3}(\R^{q+3}),g_K)$ of $q\geq 4$.  We shall consider the condition given in Proposition \ref{prop:keyN}\ {\rm (i)}-(a). Whenever $q\neq 6$,  it is known that the first standard immersion $\Phi_1$ of $M=\widetilde{G}_{3}(\R^{q+3})$ is explicitly given by one of the following immersions (see \cite[Lemma 13]{Ohnita2} for details):
\begin{enumerate}
\item[(a)] Let ${G}_{3}(\R^{q+3})\simeq SO(q+3)/S(O(q)\times O(3))$ be the non-oriented real Grassmannian manifold. Then, $\widetilde{G}_{3}(\R^{q+3})$ is isometrically immersed into $G_{3}(\R^{q+3})$ by the $2$-fold covering map $\pi: \widetilde{G}_{3}(\R^{q+3})\to {G}_{3}(\R^{q+3})$. On the other hand, $G_{3}(\R^{q+3})$ is realized by a symmetric $R$-space in some Euclidean space $\R^{N}$  by an equivariant embedding $\varphi:{G}_{3}(\R^{q+3})\to \R^{N}$.  Then, we obtain an isometric immersion of $\widetilde{G}_{3}(\R^{q+3})$ into $\R^{N}$ by the composition map $\Phi:=\varphi\circ \pi:\widetilde{G}_{3}(\R^{q+3})\to \R^{N}$.
\item[(b)] Let $\wedge^{3}(\R^{q+3})=\{v_{1}\wedge v_{2}\wedge v_{3}\mid v_{i}\in \R^{q+3}\}$ be the vector space consisting of wedge products of $3$-vectors in $\R^{q+3}$. We define the canonical inner product on  $\wedge^{3}(\R^{q+3})$, and identify $\wedge^{3}(\R^{q+3})$ with the Euclidean space $\R^{N}$ of dimension $N=\begin{pmatrix}q+3\\ 3\end{pmatrix}$. Then, $\widetilde{G}_{3}(\R^{q+3})$ is naturally immersed into $\wedge^{3}(\R^{q+3})$, and we denote the  immersion by $\Psi: \widetilde{G}_{3}(\R^{q+3})\to \wedge^{3}(\R^{q+3})$.
\end{enumerate}

As stated in \cite{Ohnita2}, both $\Phi$ and $\Psi$ provide minimal immersions $\widetilde{\Phi}, \widetilde{\Psi}$ of $ \widetilde{G}_{3}(\R^{q+3})$ into a hypersphere $S^{N-1}(r)$ of some radius $r$. Thus,  by Proposition \ref{prop:keyN} (i)-(a), a sufficient condition for the property $({\bf N}_1)$ is given by \eqref{eq:q2}. 
However, {both $\widetilde{\Phi}$ and $\widetilde{\Psi}$ do  not satisfy \eqref{eq:q2}.} Indeed, the minimal immersion $\widetilde{\varphi}: G_{3}(\R^{q+3})\to S^{N-1}(r)$ does not satisfy  \eqref{eq:q2} because there exists a stable closed geodesic in $G_{3}(\R^{q+3})$ (since $G_{3}(\R^{q+3})$ is not simply-connected). Thus, neither does the map $\widetilde{\Phi}=\widetilde{\varphi}\circ \pi$ since  $\pi$ is a covering map and  the inequality \eqref{eq:q2} is determined by a local computation of the map (More precisely, we have $\beta_\Phi=\beta_\varphi$ since $\Phi$ and $\varphi$ are equivariant under the isometric action of $SO(q+3)$). On the other hand, a direct computation shows that $\beta_{\widetilde{\Psi}}=0$. In this case, \eqref{eq:q2} is reduced to the condition \eqref{eq:cond4}, and hence, $\widetilde{\Psi}$ does not satisfy \eqref{eq:q2}.

\section{The case of $G_{2}/SO(4)$}\label{sec:g2}
In this section, we consider the case when $M=G_{2}/SO(4)$, a simply-connected exceptional compact symmetric space of rank $2$, and show the following result.
\begin{theorem}\label{thm:G}
The exceptional compact symmetric space $G_2/SO(4)$ has property $({\bf N}_1)$.
\end{theorem}

To prove this, we shall explicitly construct an immersion $\Phi$ of $G_{2}/SO(4)$ into the Euclidean space, and apply Corollary \ref{cor:cri1} to the map $\Phi$ combining with geometric properties of $\Phi$.  The description of $\Phi$ is based on geometry of symmetric R-space. We thus recall some basic facts on symmetric R-spaces and prove some lemmas. 

\subsection{Preliminaries on (symmetric) R-space}\label{subs:symR}
We refer to \cite{BCO} for details. 

 Let $U/U'$ be a Riemannian symmetric space of noncompact type, where $U$ is a connected semi-simple Lie group and $U'$ is a connected closed subgroup of $U$. We denote the Lie algebra of $U$ and $U'$ by $\fu$ and $\fu'$, respectively. The Cartan involution on $\fu$ is denoted by $\theta: \fu\to \fu$, and we write the associated Cartan decomposition by $\fu=\fu'\oplus \fp$.  As usual, we identify $\fp$ with $T_o(U/U')$.  We define an inner product $\langle,\rangle_\fu$ on $\fu$ by  $\langle X,Y\rangle_{\fu}:=-{\bf B}(X, \theta Y)$ for $X, Y\in \fu$, where  ${\bf B}$ is the Killing-form on $\fu$.  The restriction on $\fp$ defines an ${\rm Ad}_{U}(U')$-invariant inner product $\langle,\rangle_\fp$ on $\fp$. In this way, we identify $\fp$  with the Euclidean space. 
  
 We take a maximal abelian subspace $\fa$ of $\fp$. Then, we have the following direct-sum decomposition of $\fu$ (called the {\it restricted root decomposition} with respect to $\fa$):
 \begin{align*}
\fu=\fu_0\oplus \bigoplus_{\alpha\in \Sigma}\fu_{\alpha}.
\end{align*}
Here, $\alpha$ is a non-zero element in the dual space $\fa^*$ of $\fa$ such that $\fu_{\alpha}:=\{X\in \fu\mid {\rm ad}(h)X=\alpha(h)X\ \forall h\in \fa\}\neq \phi$, and $\fu_0=\fu_0'\oplus \fa$ where $\fu_0'=\{Z\in \fu'\mid [Z,h]=0\ \forall h\in\fa\}$. We call $\alpha$ a {\it root}, and denote the set of roots by $\Sigma$. It is known that $\Sigma$ becomes a root system. We define an ordering on $\fa$ by using the simple roots of $\Sigma$, and denote the set of positive roots by $\Sigma_{+}$. The restricted root decomposition satisfies the following properties:
\begin{itemize}
\item $[\fu_{\alpha}, \fu_{\beta}]\subset \fu_{\alpha+\beta}$ for any $\alpha, \beta\in \Sigma$, where $[\fu_{\alpha}, \fu_{\beta}]=0$ if $\alpha +\beta \not\in \Sigma$. 
\item Suppose $\alpha,\beta\in \Sigma\cup\{0\}$ and $\alpha\neq \beta$. Then, $\fu_\alpha$ and $\fu_\beta$ are orthogonal to each other if and only if $\beta\neq -\alpha$.
\end{itemize}


We set
$\fu'_{\alpha}:=\fu'\cap(\fu_{\alpha}\oplus\fu_{-\alpha})$, 
$\fp_{\alpha}:=\fp\cap(\fu_{\alpha}\oplus\fu_{-\alpha})$
so that 
\[
\fu'=\fu'_0\oplus \bigoplus_{\alpha\in \Sigma_+}\fu'_\alpha,\quad \fp=\fa\oplus\bigoplus_{\alpha\in \Sigma_+}\fp_{\alpha}.
\]
Note that $\fu'_\alpha=\fu'_{-\alpha}$ and $\fp_{\alpha}=\fp_{-\alpha}$ for any $\alpha\in \Sigma$. Moreover, by letting $\fp_0=\fa$, we have
\begin{align}\label{eq:brakp}
[\fu'_{\alpha}, \fu'_{\beta}]\subset \fu'_{\alpha+\beta}\oplus \fu'_{\alpha-\beta}, \quad [\fu'_{\alpha}, \fp_{\beta}]\subset \fp_{\alpha+\beta}\oplus \fp_{\alpha-\beta},\quad [\fp_{\alpha}, \fp_{\beta}]\subset \fu'_{\alpha+\beta}\oplus \fu'_{\alpha-\beta}
\end{align}
for any $\alpha, \beta\in \Sigma\cup \{0\}$ since $[\fu_{\alpha}, \fu_{\beta}]\subset \fu_{\alpha+\beta}$, $[\fu', \fu']\subset \fu'$, $[\fu',\fp]\subset \fp$ and  $[\fp,\fp]\subset \fu'$.

It is known that the isotropy representation of $U/U'$ is equivalent to ${\rm Ad}|_{U'}: U'\to O(\fp)$.
 We take an element $\eta\in \fp$ and consider the adjoint orbit through $\eta$
 \[
\mathcal{O}_\eta:=U'\cdot \eta={\rm Ad}_{U}(U')\eta.
 \]
The orbit $\mathcal{O}_\eta$ is called an {\it R-space}. We denote  the stabilizer subgroup at $\eta$ by $U''$ so that $\mathcal{O}_\eta$ is diffeomorphic to $U'/U''$. We induce a Riemannian metric on $\mathcal{O}_{\eta}$ from $\fp$ and denote the isometric embedding of $\mathcal{O}_\eta$ into $\fp$ by
 \[
 \varphi: \mathcal{O}_\eta\to \fp.
 \]
 
Suppose that $\fa$ is a maximal abelian subspace in $\fp$ containing $\eta$. Then the geometry of $\varphi$ is described by using the associated restricted root system. For example, the tangent space $T_{\eta}\mathcal{O}_\eta$ and the normal space $T^{\perp}_{\eta}\mathcal{O}_\eta$ at $\eta$  are expressed by (see \cite[Subsection 2.3.2]{BCO})
 \begin{align}\label{eq:tannor}
T_{\eta}\mathcal{O}_\eta=\bigoplus_{\alpha\in \Sigma_{>0}} \fp_{\alpha},\quad 
T^{\perp}_{\eta}\mathcal{O}_\eta=\fa\oplus \bigoplus_{\alpha\in \Sigma_{0}}\fp_{\alpha},
 \end{align}
 where we put $\Sigma_{>0}:=\{\alpha\in \Sigma_{+}\mid \alpha(\eta)>0\}$ and $\Sigma_{0}:=\{\alpha\in \Sigma_{+}\mid \alpha(\eta)=0\}$. 
According to  \eqref{eq:tannor}, we write a tangent vector $X\in T_{\eta}\mathcal{O}_\eta$ and a normal vector $\xi\in T^\perp_{\eta}\mathcal{O}_\eta$ by
 \begin{align*}
 X=\sum_{\alpha\in \Sigma_{>0}} X_{\alpha}\quad (X_{\alpha}\in \fp_{\alpha})\quad {\rm and}\quad  \xi=\xi_{0}+\sum_{\mu\in \Sigma_{0}}\xi_{\mu}\quad (\xi_{0}\in\fa,\ \xi_{\mu}\in \fp_{\mu}).
 \end{align*}

We denote the second fundamental form and the shape operator of $\varphi$ by $B_{\varphi}$ and $A_{\varphi}$, respectively.  Then, we have the following formula:
 
 \begin{proposition}[cf. \cite{TT}]\label{lem:shape}
 For any unit normal vector $\xi\in T^\perp_{\eta}\mathcal{O}_\eta$, we have 
  \begin{align}\label{eq:A0}
 (A_{\varphi})_{\xi}(X)=\sum_{\alpha\in \Sigma_{>0}}\Big{\{}-\frac{\alpha(\xi_{0})}{\alpha(\eta)}X_{\alpha}-\frac{1}{\alpha(\eta)^{2}}\sum_{\mu\in \Sigma_{0}}
 [\xi_{\mu}, [\eta, X_{\alpha}]]\Big{\}}
 \end{align}
  for any $X\in T_{\eta} \mathcal{O}_\eta$.
 \end{proposition}
 
 \begin{proof}
 We give a proof for the sake of convenience of the reader. We remark that the sign is different from \cite{TT} at some places because our description is based on the noncompact dual of  \cite{TT}. 
 
 Since $\mathcal{O}_\eta={\rm Ad}_{U}(U')\eta$, the tangent space of $\mathcal{O}_\eta$ at $\eta$ is written by $T_\eta\mathcal{O}_\eta=\{Z^*_\eta\mid Z\in \fu'\}$, where 
 \begin{align}\label{eq:fv}
 Z^*_w:=\frac{d}{dt}{\rm Ad}({\rm exp}tZ)w\Big{|}_{t=0}={\rm ad}(Z)w=[Z,w]\quad \textup{ for $w\in \fp$}.
 \end{align}
  Note that $Z^*$ defines a Killing vector field on $\fp$ since ${\rm Ad}_{U}(U')\subset O(\fp)$. Fix an arbitrary $X\in T_\eta\mathcal{O}_\eta$, and we take $Z\in \fu'$ so that $Z^*_\eta=[Z,\eta]=X$. More precisely, $Z$ is given by
 \begin{align}\label{eq:z}
 Z=\sum_{\alpha\in \Sigma_{>0}} -\frac{1}{\alpha(\eta)^2}[\eta, X_\alpha].
 \end{align}
Indeed, we see $Z\in \fu'$ since $[\fp,\fp]\subset \fu'$ and a direct computation shows that $[Z,\eta]=X$.

Take any $Y\in T_\eta\mathcal{O}_\eta$ and $\xi\in T_\eta^\perp\mathcal{O}_\eta$, and we extend $\xi$ so that $\langle Z^*,\xi\rangle=0$ on a neighborhood of $\eta$ in $\fp$. Then, we have
\begin{align}\label{eq:shape1}
\langle (A_{\varphi})_\xi(X),Y\rangle_{\eta}=\langle (A_{\varphi})_\xi(Y),X\rangle_{\eta}=-\langle \overline{\nabla}_Y\xi, Z^*\rangle_{\eta}=\langle \overline{\nabla}_YZ^*,\xi\rangle_{\eta}=-\langle \overline{\nabla}_\xi Z^*,Y\rangle_{\eta},
\end{align}
where $\overline{\nabla}$ is the Levi-Civita connection on $\fp$, and we used the fact that $Z^*$ is a Killing vector field on $\fp$, which implies $\langle \overline{\nabla}_uZ^*,v\rangle+\langle \overline{\nabla}_vZ^*,u\rangle=0$ for any $u,v\in T_\eta\fp$.  Since $Y$ is arbitrary,  \eqref{eq:shape1} shows that
$
(A_{\varphi})_{\xi}(X)=-(\overline{\nabla}_\xi Z^*)^\top.
$
Moreover, by taking a curve $c:(-\epsilon,\epsilon)\to \fp$ satisfying that $c(0)=\eta$ and $c'(0)=\xi$, we see
\begin{align*}
\overline{\nabla}_\xi Z^*=\frac{d}{ds} Z^{*}(c(s))\Big{|}_{s=0}=\frac{d}{ds}\frac{d}{dt}{\rm Ad}({\rm exp} tZ)c(s)\Big{|}_{s=t=0}=\frac{d}{dt}{\rm Ad}({\rm exp}tZ)\xi\Big{|}_{t=0}=[Z,\xi].
\end{align*}
Therefore, we obtain 
\[
(A_{\varphi})_{\xi}(X)=[\xi,Z]^{\top}.
\]
Here, we have
\begin{align}\label{eq:zxi}
[\xi,Z]&=\sum_{\alpha\in \Sigma_{>0}}\Big{\{}-\frac{1}{\alpha(\eta)^2}[\xi_0, [\eta, X_{\alpha}]]-\frac{1}{\alpha(\eta)^2}\sum_{\mu\in\Sigma_0}[\xi_u,[\eta,X_\alpha]]\Big{\}}.
\end{align}
since $Z$ is given by \eqref{eq:z}.  In \eqref{eq:zxi}, we shall show
\begin{align*}
[\xi_0,[\eta, X_{\alpha}]]=\alpha(\xi_0)\alpha(\eta)X_\alpha\quad {\rm and}\quad [\xi,Z]=[\xi,Z]^{\top}
\end{align*}
which implies the desired formula \eqref{eq:A0}.
Indeed, by definition of $\fp_{\gamma}$ ($\gamma\in \Sigma_+$), any element $X_{\gamma}\in \fp_\gamma$ is written by
$
X_\gamma=\widetilde{X}_\gamma+\widetilde{X}_{-\gamma}
$
for some elements $\widetilde{X}_{\gamma}\in \fu_{\gamma}$ and $\widetilde{X}_{-\gamma}\in \fu_{-\gamma}$. Thus, we have
\begin{align*}
[h_1, X_\gamma]=\gamma(h_1)(\widetilde{X}_\gamma-\widetilde{X}_{-\gamma}),\quad [h_2,[h_1, X_\gamma]]=\gamma(h_2)\gamma(h_1)X_\gamma
\end{align*}
for any $h_1,h_2\in \fa$. This implies $[\xi_0,[\eta, X_{\alpha}]]=\alpha(\xi_0)\alpha(\eta)X_\alpha$.
On the other hand, we have that $[\xi_\mu,[\eta, X_\alpha]]\in \fp_{\mu+\alpha}\oplus \fp_{\mu-\alpha}$ by \eqref{eq:brakp}. Since $\mu\in \Sigma_0$ and $\alpha\in \Sigma_{>0}$, we see $(\mu\pm \alpha)(\eta)\neq 0$, and hence, $[\xi_\mu,[\eta, X_\alpha]]\in T_\eta\mathcal{O}_\eta$. Therefore,  we see $[\xi,Z]\in T_\eta\mathcal{O}_\eta$, i.e. $[\xi,Z]=[\xi,Z]^{\top}$ as required.
  \end{proof}
 
If $X$ belongs to a special subspace in $T_{\eta}\mathcal{O}_\eta$,  we obtain the following simple formulas:
 
 \begin{lemma}\label{lem:keyR}
 Let $S$ be a subset of $\Sigma_{>0}$ and set 
\[
\fp_S:=\bigoplus_{\alpha\in S}\fp_\alpha\subset T_\eta \mathcal{O}_\eta.
\]
Suppose $S$ consists of strongly orthogonal roots i.e. $\alpha\pm\beta\notin \Sigma$ for any $\alpha, \beta\in S$. Then we have
\begin{align}\label{eq:R1}
B_\varphi(X,Y)=\sum_{\alpha\in S}- \frac{\langle X_\alpha, Y_\alpha\rangle}{\alpha(\eta)}\alpha^\sharp\in\fa\quad \textup{for any $X,Y\in \fp_S$},
\end{align}
 where $\alpha^\sharp$ denotes the dual vector of the root $\alpha$ with respect to the inner product on $\fa$. 

Furthermore, we have
\begin{align}\label{eq:R2}
|B_\varphi(X,X)|^2=\sum_{\alpha\in S}\frac{|\alpha|^{2}}{\alpha(\eta)^{2}} |X_\alpha|^4\quad \textup{for any $X\in \fp_{S}$.}
\end{align}
 \end{lemma}
\begin{proof}
Suppose $S$ consists of strongly orthogonal roots. Fix arbitrary elements $\alpha,\beta\in S\subset \Sigma_{>0}$ (possibly $\alpha=\beta$)  and $\mu\in \Sigma_0$.  Then, for any $X_\alpha\in \fp_\alpha$, $Y_\beta\in \fp_\beta$ and $\xi_\mu\in \fp_\mu$, we have
\begin{align*}
\langle [\xi_\mu, [\eta, X_\alpha]], Y_\beta\rangle_{\fp}={\bf B}([\xi_\mu, [\eta, X_\alpha]], Y_{\beta})=-{\bf B}([\eta,X_\alpha], [\xi_\mu, Y_\beta])=\langle [\eta,X_\alpha], [\xi_\mu, Y_\beta]\rangle_{\fu}
\end{align*}
since the Killing form ${\bf B}$ is ${\rm Ad}(U)$-invariant. By properties of root decomposition, we have
\begin{align*}
&[\eta, X_\alpha]\in \fu_{\pm \alpha}\quad {\rm and}\quad [\xi_\mu,Y_\beta]\in \fu_{\pm(\mu+\beta)}\oplus\fu_{\pm(\mu-\beta)},
\end{align*}
where we put $\fu_{\pm \gamma}=\fu_\gamma\oplus \fu_{-\gamma}$. Since $\alpha\pm \beta\not\in \Sigma$ by assumption, we have $\mu\pm \beta\neq \pm \alpha$. Therefore, $[\eta, X_\alpha]$ and $[\xi_\mu, Y_\beta]$ are orthogonal to each other, and hence, we obtain 
\[
\langle [\xi_\mu, [\eta, X_\alpha]], Y_\beta\rangle_{\fp}=0.
\]

 In particular, for any $X, Y\in\fp_{S}$ and $\xi\in T^\perp_{\eta}\mathcal{O}_\eta$, the formula \eqref{eq:A0} implies that 
\begin{align*}
\langle (A_{\varphi})_{\xi}(X),Y\rangle=\sum_{\alpha\in S}-\frac{\alpha(\xi_0)}{\alpha(\eta)} \langle X_\alpha, Y\rangle=\sum_{\alpha\in S}-\frac{\alpha(\xi_0)}{\alpha(\eta)} \langle X_\alpha, Y_\alpha\rangle.
\end{align*}
 Therefore, by using an orthonormal basis $\{h_{k}\}_{k}$ of $\fa$,   we see
\[
B_\varphi(X, Y)=\sum_{k}\langle A^\varphi_{h_k}(X),Y\rangle h_k=\sum_k\sum_{\alpha\in S}-\frac{\alpha(h_k)}{\alpha(\eta)}\langle X_\alpha,  Y_\alpha\rangle h_k=\sum_{\alpha\in S}-\frac{\langle X_\alpha,  Y_\alpha\rangle}{\alpha(\eta)}\alpha^\sharp
\]
for any $X, Y\in\fp_{S}$. This proves the formula \eqref{eq:R1}.

Furthermore,  we have 
\[
|B_{\varphi}(X,X)|^{2}=\sum_{\alpha,\beta\in S}\frac{|X_{\alpha}|^{2}|X_{\beta}|^{2}}{\alpha(\eta)\beta(\eta)}\langle \alpha,\beta\rangle.
\]
Since $\Sigma$ is an abstract root system in the sense of \cite{Knapp} (see \cite[Corollary 6.53]{Knapp}),  Proposition 2.48 (f) in \cite{Knapp} shows that if $\alpha\pm\beta \notin \Sigma\cup \{0\}$, then $\langle \alpha,\beta\rangle=0$. Because we suppose $\alpha\pm\beta\not\in \Sigma$ for any $\alpha,\beta\in S\subset \Sigma_{>0}$, $\langle \alpha,\beta\rangle\neq 0$ if and only if $\alpha=\beta$.
  Therefore,  we obtain \eqref{eq:R2}.
\end{proof}
 
In the following, we restrict our attention to a special element in $\fp$. Suppose that $\eta$ is an element in $\fp$ such that ${\rm ad}(\eta): \fu\to \fu$ has exactly three eigenvalues $0$, $1$ and $-1$. We denote the $0$, $1$, $-1$-eigenspace by $\fu^{0}$, $\fu^{1}$, and $\fu^{-1}$, respectively. 
Note that the eigen-decomposition
 $
\fu=\fu^{-1}\oplus \fu^{0}\oplus \fu^{1}
$ gives a gradation of $\fu$, i.e. $[\fu^k, \fu^l]\subseteq \fu^{k+l}$ for any $k,l\in \mathbb{Z}$, where $\fu^k=\{0\}$ if $k\neq 0,1,-1$.

Then, the orbit $\mathcal{O}_\eta\simeq U'/U''$ becomes a {\it symmetric R-space}, namely, $\mathcal{O}_{\eta}$ is a Riemannian symmetric space with respect to the induced metric (see \cite[Example 2.8.3]{BCO} and \cite{TK}). We denote the Lie algebra of the isotropy subgroup $U''$ by $\fu''$.  Then, we obtain a reductive decomposition 
\[
\fu'=\fu''\oplus \fp',\quad \textup{where $\fu''=\fu'\cap\fu^{0}$ and $\fp'=\fu'\cap (\fu^{-1}\oplus \fu^{1})$.}
\]
Indeed, we have $\fu''=\fu'\cap\fu^{0}$ since ${\rm Ad}_{U}(U'')\eta=\eta$. Moreover, for any $X\in \fu'$,  $X$ is decomposed into $X=X^0+X^{-1}+X^1$ where $X^k\in \fu^k$. Then, we have ${\rm ad}(\eta)^2X=X^{-1}+X^1$.  Since $[\fp, \fp]\subset \fu'$, $[\fu',\fp]\subset \fp$ and $\eta\in \fp$, we see ${\rm ad}(\eta)^2\fu'\subset \fu'$, and hence, $X^{-1}+X^1\in \fu'$ and also $X_0\in \fu'$. Therefore, we obtain a decomposition $\fu'=(\fu'\cap \fu^{0})\oplus (\fu'\cap (\fu^{-1}\oplus \fu^{1}))=\fu''\oplus \fp'$. Since this decomposition is orthogonal, it follows that $\fp'$ is an ${\rm Ad}_U(U'')$-invariant subspace, and hence, $\fu'=\fu''\oplus \fp'$ is a reductive decomposition for $U'/U''$.
 As usual, we identify $\fp'$ with $T_{o}(U'/U'')$ via the map $X\mapsto \frac{d}{dt}{\rm exp}tX\cdot o|_{t=0}$.

Since $U'/U''$ is identified with $\mathcal{O}_\eta$ by the map $F: [u']\mapsto {\rm Ad}(u')\eta$ ($u'\in U'$),  the tangent vector $X\in \fp'$ corresponds to $dF_o(X)=\frac{d}{dt}{\rm Ad}({\rm exp}(tX))\eta|_{t=0}=[X,\eta]=-{\rm ad}(\eta)X$ in $T_\eta\mathcal{O}_\eta\subset \fp$. Namely, we have an isomorphism
\[
-{\rm ad}(\eta)|_{\fp'}: \fp'\to T_{\eta}\mathcal{O}_\eta.
\]

\begin{lemma}\label{prop:keyR}
Suppose $\mathcal{O}_\eta\simeq U'/U''$ is a symmetric R-space. Then we have the following:
\begin{enumerate}
\item $-{\rm ad}(\eta)|_{\fp'}: \fp'\to T_{\eta}\mathcal{O}_\eta$ is an ${\rm Ad}_{U}(U'')$-equivariant isomorphism satisfying that $[-{\rm ad}(\eta)X, -{\rm ad}(\eta)Y]=-[X,Y]$ for any $X,Y\in \fp'$. In particular, there exists a 1-1 correspondence between abelian subspaces in $\fp'$ and abelian subspaces in $T_\eta \mathcal{O}_\eta\subset \fp$. 
\item  $\alpha(\eta)=1$ for any $\alpha\in \Sigma_{>0}$.
\end{enumerate}
\end{lemma}
\begin{proof}
(i) For any $X\in \fp'$ and  $u''\in U''$, we see
\[
{\rm Ad}(u'')\circ (-{\rm ad}(\eta))X={\rm Ad}(u'')[X,\eta]=[{\rm Ad}(u'')X, {\rm Ad}(u'')\eta]=[{\rm Ad}(u'')X, \eta]=-{\rm ad}(\eta)\circ {\rm Ad}(u'')X.
\]
 This shows $T_\eta \mathcal{O}_\eta$ is also an $U''$-invariant subspace and $-{\rm ad}(\eta)|_{\fp'}$ is $U''$-equivariant.

Since $\fu=\fu^{-1}\oplus \fu^{0}\oplus \fu^{1}$ is a gradation of $\fu$, we have that $[\fu^{1},\fu^{1}]=[\fu^{-1},\fu^{-1}]=\{0\}$.  In particular, for any $X,Y\in \fu^{-1}\oplus\fu^{1}$, it is easy to see
$
[{\rm ad}(\eta)X, {\rm ad}(\eta)Y]=-[X,Y].
$
Because $\fp'\subset \fu^{-1}\oplus \fu^{1}$,  this shows that we have $[-{\rm ad}(\eta)X, -{\rm ad}(\eta)Y]=-[X,Y]$ for any $X,Y\in \fp'$. 
In particular, if $\fa'$ is an abelian subspace in $\fp'$, then $-{\rm ad}(\eta)(\fa')$ is also an abelian subspace in $T_\eta \mathcal{O}_\eta$, and this correspondence is invertible.  This proves (i).

(ii) For any $\alpha\in \Sigma_{>0}$, we take a non-zero element $X_\alpha\in \fp_\alpha$, and decompose $X_{\alpha}$ into $X_\alpha=X_\alpha^{-1}+X_\alpha^{0}+X_\alpha^{1}$ with respect to the eigen-decomposition $\fu=\fu^{-1}\oplus \fu^{0}\oplus \fu^{1}$. Since $X_\alpha\in \fp_\alpha$, we have ${\rm ad}(\eta)^2X_\alpha=\alpha(\eta)^2X_\alpha$ and this implies that $X_\alpha^{-1}+X_\alpha^{1}=\alpha(\eta)^2(X_\alpha^{-1}+X_\alpha^{0}+X_\alpha^{1})$. Because $\alpha(\eta)>0$,  we obtain  $X_\alpha^0=0$ and $\alpha(\eta)=1$, as required.
\end{proof}

Finally, we prove an algebraic lemma  which we use later.

\begin{lemma}\label{lem:keyR1}
Suppose $\mathcal{O}_\eta$ is a symmetric $R$-space.  Let $\fa_0$ be a two-dimensional subspace in $T_\eta \mathcal{O}_\eta$ 
which is spanned by $\{e_1,e_2\}$ of the form
\begin{align*} 
e_1=P_1+P_2\ or\ P_1-P_2,\quad e_2=P_2+P_3
\end{align*}
for some unit vector $P_i\in \fp_{\alpha_i}$ $(i=1,2,3)$, where $\alpha_{1}$, $\alpha_{2}$ and $\alpha_{3}$ are elements in $\Sigma_{>0}$ that are not proportional to each other. Moreover, we suppose
\begin{enumerate}
\item[(a)] $\alpha_i-\alpha_j\notin \Sigma$ for any $i,j=1,2,3$.
\item[(b)] $|\alpha_1|=|\alpha_2|=|\alpha_3|=D$ for some constant $D>0$.
\end{enumerate}
Then, we have
\[
|B_\varphi(v,v)|^2=\frac{D^2}{2}=(\textup{const.})\quad \textup{\it for any  unit vector $v\in \fa_0$}.
\]
\end{lemma}

\begin{proof}
It suffices to prove in the case when $e_1=P_1+P_2$ and $e_2=P_2+P_3$. We first note that, by Lemma \ref{prop:keyR}-(ii), 
$\alpha+\beta\notin \Sigma$ for any $\alpha,\beta\in \Sigma_{>0}$. Thus, we may assume ${\rm (a)}'$ $\alpha_i\pm \alpha_j\notin \Sigma$ for any $i,j=1,2,3$, i.e. $\{\alpha_i\}_{i=1}^3$ consists of strongly orthogonal roots.  In particular, we are able to apply Lemma \ref{lem:keyR} to the subset $S=\{\alpha_i\}_{i=1}^3$. By  condition (b) and Lemma \ref{prop:keyR}-(ii), the formula \eqref{eq:R2} becomes
\[
|B_\varphi(v,v)|^2=\sum_{i=1}^{3} D^{2} |v_{\alpha_{i}}|^4
\]
for any $v\in \fp_{S}$. Since $v\in \fa_{0}\subset \fp_{S}$, we may write $v=v_{1}e_{1}+v_{2}e_{2}=v_{1}P_{1}+(v_{1}+v_{2})P_{2}+v_{2}P_{3}$ for some $v_{1}, v_{2}\in \R$. Because $\alpha_{i}\neq \alpha_{j}$, we have $v_{\alpha_{1}}=v_{1}P_{1}$, $v_{\alpha_{2}}=(v_{1}+v_{2})P_{2}$ and $v_{\alpha_{3}}=v_{2}P_{3}$.
 In particular, we see
\[
|v|^{2}=1\quad \Longleftrightarrow\quad  v_{1}^{2}+(v_{1}+v_{2})^{2}+v_{2}^{2}=1\quad \Longleftrightarrow\quad v_{1}^{2}+v_{2}^{2}+v_{1}v_{2}=\frac{1}{2}.
\]
Thus, for any unit vector $v\in \fa_{0}$, we see
\begin{align*}
|B_\varphi(v,v)|^2=\sum_{i=1}^{3} D^{2} |v_{\alpha_{i}}|^4=D^{2}\cdot \{v_{1}^{4}+(v_{1}+v_{2})^{4}+v_{2}^{4}\}=D^{2}\cdot 2(v_{1}^{2}+v_{2}^{2}+v_{1}v_{2})^{2}=\frac{D^{2}}{2}
\end{align*}
as required.
\end{proof}

\subsection{The immersion of $G_{2}/SO(4)$} 
Let us consider a specific situation. Suppose $U/U'=SL(7, \R)/SO(7)$. Note that ${\rm dim}\ U/U'=27$ and ${\rm rank}\ U/U'=6$. The associated Cartan involution is given by $\theta: \fu\to \fu$, $X\mapsto -{}^{t}X$, and the eigen-decomposition with respect to $\theta$ yields the Cartan decomposition $\fu=\fu'\oplus \fp$, where
\begin{align*}
\fu&=\mathfrak{sl}(7, \R)=
\{X\in \mathfrak{gl}(7,\R)\mid {\rm tr}X=0\}, \\
\fu'&=\mathfrak{o}(7)=
\{X\in \mathfrak{gl}(7,\R)\mid  {}^{t}X=-X\},\\
\fp&=
\{X\in \mathfrak{gl}(7,\R)\mid  {\rm tr}X=0,\ {}^{t}X=X\}.
\end{align*}

The Killing form ${\bf B}$ on $\fu$ is given by ${\bf B}(X,Y)=2\cdot 7\cdot {\rm tr}_{\R}(XY)$ for $X, Y\in \fu$. This induces the canonical inner product $\langle\cdot,\cdot \rangle_{\fp}$ on $\fp$ defined by
\[
\langle X, Y\rangle_{\fp}:=-{\bf B}(X,\theta Y)=C\cdot {\rm tr}_{\R}(XY)
\]
for $X, Y\in \fp$, where we put $C:=2\cdot 7=14$.

We consider an orbit $\mathcal{O}_\eta={\rm Ad}_U(U')\cdot \eta$ through a specific element 
\begin{align}\label{def:eta}
{\eta}:=\left[
\begin{array}{c|ccc}
\frac{4}{7}I_{3} & O \\ \hline
O& -\frac{3}{7}I_{4}
\end{array}
\right]\in \fp,
\end{align}
where ${I}_{n}$ denotes the $n\times n$ identity matrix. Note that $|\eta|_{\fp}^2=24$. It is easy to check that the eigenvalues of ${\rm ad}(\eta): \fu\to \fu$ are $0, -1, 1$, and hence, the orbit $\mathcal{O}_\eta$ becomes a {symmetric R-space}.  Moreover, we see that the isotropy subgroup at $\eta$ is given by $S(O(3)\times O(4))$, namely, $\mathcal{O}_\eta$ is isometric to the symmetric space $SO(7)/S(O(3)\times O(4))$, which is diffeomorphic to the non-oriented real Grassmannian manifold $G_3(\R^7)$. We identify the orbit $\mathcal{O}_\eta$ with $G_3(\R^7)$. Then, we obtain an embedding
\[
\varphi: G_{3}(\R^{7})\to \fp.
\]

Next, we shall construct an immersion of $M=G_{2}/SO(4)$.  By definition, $G_{2}$ is the automorphism group of the octonion algebra ${\bf O}$.  It is known that $G_{2}$ is actually acts on ${\rm Im}({\bf O})\simeq \R^{7}$ and $G_{2}$ is a closed subgroup of $SO(7)$. In particular, $G_{2}$ is naturally acts on the oriented $3$-plane real Grassmannian manifold $\widetilde{G}_{3}(\R^{7})$, and it turns out that the isotropy subgroup at the origin $o\in \widetilde{G}_{3}(\R^{7})$ is isomorphic to $SO(4)$ (see \cite[Theorem 1.8 in Ch. IV]{HL}. Note that the orbit $G_2\cdot o$ is diffeomorphic to the associative Grassmannian). Therefore, we have a diffeomorphism $ G_2\cdot o\simeq G_2/SO(4)$ and obtain an equivariant embedding $\iota': G_{2}/SO(4)\to \widetilde{G}_{3}(\R^{7})$. Moreover, by regarding $\widetilde{G}_{3}(\R^{7})$ as a symmetric space, one easily verifies that $\iota'$ is a totally geodesic embedding (see also the next section for details). In particular, $G_2\cdot o$ is also a symmetric space with respect to the induced metric, and hence, $\iota'$ is an isometric embedding of the symmetric space $G_{2}/SO(4)$ into $\widetilde{G}_{3}(\R^{7})$. By composing the 2-fold covering map $\pi:\widetilde{G}_{3}(\R^{7})\to G_{3}(\R^{7})$, we obtain a totally geodesic immersion
\[
\iota=\pi\circ \iota': G_{2}/SO(4) \to G_{3}(\R^{7}).
\]
Now, we set
\[
\Phi=\varphi\circ\iota: G_{2}/SO(4)\to \fp\simeq \R^{27}.
\]
Since $\varphi$ is a map into the hypersphere $S^{26}(|\eta|)\subset \fp$, we also obtain an immersion $\widetilde{\Phi}: G_{2}/SO(4)\to S^{26}(|\eta|)$.
 Then, $\Phi$ and $\widetilde{\Phi}$ have the following nice geometric properties:

\begin{proposition}\label{prop:keyg}
We have the following:
\begin{enumerate}
\item $\widetilde{\Phi}$ is a minimal immersion into the hypersphere $S^{26}(|\eta|)$. 
\item The second fundamental form $B_\Phi$ of $\Phi$ is isotropic, i.e. $|B_\Phi(v,v)|$ is constant for any unit tangent vector $v$ of $G_2/SO(4)$.
\end{enumerate}
\end{proposition}
\begin{remark}
{\rm
The embedding $\varphi: G_{3}(\R^{7})\to \fp$ also provides a minimal immersion $\widetilde{\varphi}: G_{3}(\R^{7})\to S^{26}(|\eta|)$ by the result of Takeuchi-Kobayashi \cite[Theorem 4.2]{TK}. However, (i) is not a direct consequence of this fact.  We also remark that the second fundamental form of $\varphi: G_{3}(\R^{7})\to \fp$ is not isotropic.
}
\end{remark}
As a consequence of Proposition \ref{prop:keyg}, we see that $|B_{\Phi}(v,v)|$ coincides with $\beta_{\Phi}$ for any unit tangent vector $v$ of $G_2/SO(4)$, and hence, we can easily determine the value. The explicit value of $\beta_{\Phi}$ will be used in the proof of our main result (Theorem \ref{thm:G}). 

\subsection{Proof of Proposition \ref{prop:keyg}}
In this subsection, we give a proof of Proposition \ref{prop:keyg}. We use geometry of symmetric R-space described in Section \ref{subs:symR}. Suppose $U/U'=SL(7,\R)/SO(7)$, and denote the Cartan decomposition by $\fu=\fu'\oplus \fp$. A maximal abelian subspace $\fa$ of $\fp$ is given by
\[
\fa={\rm span}_{\R}\{H_{i}:=E_{i,i}-E_{i+1,i+1}\mid 1\leq i\leq 6\},
\]
where $E_{i,j}$ is the $(i,j)$-matrix unit. Note that $\{H_i\}_{i=1}^6$ is {\it not} an orthogonal basis.

The subspace $\fa$ contains the element $\eta$ given by \eqref{def:eta}. Indeed, we see
\begin{align}\label{eq:eta}
\eta=\frac{1}{7}(4H_{1}+8H_{2}+12H_{3}+9H_{4}+6H_{5}+3H_{6}).
\end{align}
Then, the associated restricted root system and positive roots are given by 
\[
\Sigma=\{\epsilon_{i}-\epsilon_{j}\mid i\neq j\},\quad \Sigma_{+}=\{\epsilon_{i}-\epsilon_{j}\mid 1\leq i <j\leq 7\},
\]
where $\epsilon_{i}\in \fa^{*}$ ($i=1,\ldots, 7$) is defined by 
\[
\epsilon_{i}(A):=(\textup{$(i,i)$-component of $A$})\quad {\rm for}\quad A\in \fa.
\]
Note that $\Sigma$ is a root system of type ${A}_{6}$ and  the set of simple roots of $\Sigma$ is given by $\{\epsilon_{i}-\epsilon_{i+1}\}_{i=1}^{6}$. The root vector of the simple root with respect to the induced metric on $\fa$ is given by
\begin{align}\label{eq:simpr}
(\epsilon_{i}-\epsilon_{i+1})^{\sharp}=\frac{1}{C}H_{i}
\end{align}
for each $i=1,\ldots, 6$, where the {\it root vector} $\alpha^{\sharp}$ of $\alpha$ means the element in $\fa$ defined by the relation $\alpha=\langle \alpha^{\sharp},\cdot \rangle|_{\fa}$.

Moreover, each root space is expressed by
\[
\fp_{\epsilon_{i}-\epsilon_{j}}={\rm span}_{\R}\{P_{ij}:=E_{i,j}+E_{j,i}\}.
\]
Note that $|P_{ij}|^{2}=2C$ for any $1\leq i<j\leq 7$.

We put $I:=\{(i,j)\mid 1\leq i<j\leq 7\}$ and 
\[I_{0}:=\{(i,j)\in I\mid 1\leq i<j\leq 3\ {\rm or}\ 4\leq i<j\leq 7\}\]
so that 
\[
I\setminus I_{0}=\{(i,j)\mid 1\leq i\leq 3\ {\rm and}\ 4\leq j \leq 7\}.
\]
 Then, for the element $\eta$ defined by \eqref{def:eta}, we have
\begin{align*}
\Sigma_{>0}=\{\epsilon_{i}-\epsilon_{j}\mid (i,j)\in I\setminus I_{0}\},\quad \Sigma_{0}=\{\epsilon_{i}-\epsilon_{j}\mid  (i,j)\in I_0\}.
\end{align*}
In particular, by \eqref{eq:tannor}, the tangent space and the normal space of $\varphi: \mathcal{O}_\eta\to \fp$ are expressed by
\begin{align*}
T_\eta  \mathcal{O}_\eta&=
 \Big{\{}
\left[
\begin{array}{c|ccc}
O &B \\ \hline
 {}^{t}B &  O
\end{array}
\right]\in \mathfrak{gl}(7,\R)\mid
B\in M(3,4,\R)
\Big{\}},\\
T_\eta^{\perp}  \mathcal{O}_\eta&=
 \Big{\{}
\left[
\begin{array}{c|ccc}
A & O \\ \hline
O&  C
\end{array}
\right]\in \mathfrak{gl}(7,\R)\mid
A\in {\rm Sym}(3,\R),\ C\in {\rm Sym}(4,\R),\ {\rm tr}A+{\rm tr}C=0
\Big{\}},
\end{align*} 
where ${\rm Sym}(n,\R)$ is the set of $n\times n$ symmetric matrices.

Recall that the symmetric R-space $\mathcal{O}_\eta$ is isometric to the symmetric space $U'/U'':=SO(7)/S(O(3)\times O(4))\simeq G_{3}(\R^{7})$.
The associated reductive decomposition $\fu'=\fu''\oplus \fp'$ of $U'/U''$ (see Section \ref{subs:symR}) is precisely given by 
\begin{align*}
\fu'&=\fo(7)=\{X\in \mathfrak{gl}(7,\R)\mid {}^{t}X=-X\},\\
\fu''&=\fu'\cap \fu^0=\Big{\{}
\left[
\begin{array}{c|ccc}
A & O\\ \hline
O &  C
\end{array}
\right]\in \mathfrak{gl}(7,\R)\mid
A\in \fo(3),\ C\in \fo(4)
\Big{\}},\\
\fp'&=\fu'\cap (\fu^{-1}\oplus \fu^{1})
=\Big{\{}
\left[
\begin{array}{c|ccc}
O &B \\ \hline
 -{}^{t}B &  O
\end{array}
\right]\in \mathfrak{gl}(7,\R)\mid
B\in M(3,4,\R)
\Big{\}}.
\end{align*}
Note that this coincides with the usual Cartan decomposition associated to $U'/U''$.
By Lemma \ref{prop:keyR} (i), we have an ${\rm Ad}_U(U'')$-equivariant isomorphism $-{\rm ad}(\eta)|_{\fp'}: \fp'\to T_\eta\mathcal{O}_\eta$. More precisely, this isomorphism is given by
\begin{align}\label{eq:corre}
\fp'\ni
\left[
\begin{array}{c|ccc}
O &B \\ \hline
 -{}^{t}B &  O
\end{array}
\right]
\mapsto 
\left[
\begin{array}{c|ccc}
O &-B \\ \hline
 -{}^{t}B &  O
\end{array}
\right]
\in T_{\eta}\mathcal{O}_\eta.
\end{align}

Next, we consider the immersion $\Phi=\varphi\circ \iota: G_{2}/SO(4)\to \fp$.
Recall that $G_{2}$ is a closed subgroup of $SO(7)$. In fact, the Lie algebra $\fg_2$ of $G_2$ is expressed as a Lie subalgebra of $\fo(7)$ as follows: 
We set 
\[
G_{ij}:=E_{i,j}-E_{j,i}\quad 1\leq i<j\leq 7,
\]
which forms a basis of $\fo(7)$. Then we have the following:
\begin{lemma}[Theorem 1.4.3 in \cite{Yokota}]
The Lie algebra $\fg_{2}$ is expressed by 
\begin{align*}
\fg_{2}=\left\{
\sum_{1\leq i<j\leq 7} \lambda_{ij}G_{ij}\;\left|\;
  \begin{gathered}
\lambda_{23}+\lambda_{45}+\lambda_{67}=0,\quad -\lambda_{13}-\lambda_{46}+\lambda_{57}=0,\\
\lambda_{12}+\lambda_{47}+\lambda_{56}=0,\quad -\lambda_{15}+\lambda_{26}-\lambda_{37}=0,\\
\lambda_{14}-\lambda_{27}-\lambda_{36}=0,\quad -\lambda_{17}-\lambda_{24}+\lambda_{35}=0,\\
\lambda_{16}+\lambda_{25}+\lambda_{34}=0,\quad \lambda_{ij}\in \R.
  \end{gathered}
\right.
\right\}\subset \fo(7).
\end{align*}
\end{lemma}

In particular, we can take a basis $\{V_{1},\ldots, V_{14}\}$ of $\fg_{2}$ as follows (cf. \cite{Sasaki}):
\begin{align*}
V_1&=G_{23}-G_{67},\quad V_2=G_{45}-G_{67},\quad V_3=G_{13}+G_{57},\quad V_4=G_{46}+G_{57},\quad V_5=G_{12}-G_{56},\\
V_6&=G_{47}-G_{56},\quad V_7=G_{15}-G_{37},\quad V_8=G_{26}+G_{37},\quad V_9=G_{14}+G_{36}, \quad V_{10}=G_{27}-G_{36},\\
V_{11}&=G_{17}+G_{35},\quad V_{12}=G_{24}+G_{35},\quad V_{13}=G_{16}-G_{34}, \quad V_{14}=G_{25}-G_{34}.
\end{align*}

We set $M=G/K=G_2/SO(4)$. Then, $M=G/K$ is immersed into the symmetric space $U'/U'':=SO(7)/S(O(3)\times O(4))\simeq G_{3}(\R^{7})$ by the composition map $\iota=\pi\circ \iota'$ of the natural embedding $\iota': G_2/SO(4)\to SO(7)/SO(3)\times SO(4)\simeq \widetilde{G}_3(\R^7)$ and the covering map $\pi: \widetilde{G}_3(\R^7)\to G_3(\R^7)$. Since our computation is local, we shall use an identification between a neighborhood around  the origin in $\widetilde{G}_3(\R^7)$ and a neighborhood around the origin in ${G}_3(\R^7)$ in the following computation.

Denote the Lie algebra of $G$ and $K$ by $\fg$ and $\fk$, respectively. Since $K$ is contained in $U''$, we have $\fk=\fg\cap \fu''$. More precisely, by using the above basis of $\fg=\fg_2$, we see
\[
\fk=\fg\cap \fu''={\rm span}_\R\{V_1,\ldots, V_6\}.
\]
Hence, by letting
\[
\fm:= \fg\cap \fp'={\rm span}_{\R}\{V_{7},\ldots, V_{14}\},
\]
we obtain a reductive decomposition $\fg=\fk\oplus \fm$ of $M=G/K$ and $\fm$ is identified with $T_oM$. Moreover, we see that the restriction $\theta|_{\fg}$ of the Cartan involution $\theta: \fu'\to\fu'$ associated with $U'/U''$ gives the Cartan involution on $\fg$, and the reductive decomposition $\fg=\fk\oplus \fm$ coincides with the Cartan decomposition of $\fg$.
It is easy to see that $\fm$ is a Lie triple system in $\fp'$, i.e. $\fm$ satisfies that $[\fm,[\fm,\fm]]\subset \fm$. In particular, $G_{2}\cdot o\simeq G_{2}/SO(4)$ is a totally geodesic submanifold in $\widetilde{G}_{3}(\R^{7})$, and $\iota: G/K\to U'/U''$ is a totally geodesic immersion. Furthermore, $M=G/K$ is also a symmetric space with respect to the induced metric.

By using the correspondence \eqref{eq:corre}, we obtain a linear isomorphism
\[
\fm\simeq \widetilde{\fm}:={\rm span}_{\R}\{\widetilde{V}_{7},\ldots \widetilde{V}_{14}\}\subset T_\eta\mathcal{O}_\eta
\]
where $\widetilde{V}_{k}$ is defined by replacing each $G_{ij}=E_{i,j}-E_{j,i}$ to $P_{ij}=E_{i,j}+E_{j,i}$ in the definition of $V_{k}$. For example, $\widetilde{V}_{7}:=P_{15}-P_{37}$, $\widetilde{V}_{8}:=P_{26}+P_{37}$, etc.  Note that the isomorphism $\fm\simeq \widetilde{\fm}$ is ${\rm Ad}_U(K)$-equivariant since $K\subset U''$ and the isomorphism $\fp'\simeq T_\eta\mathcal{O}_\eta$ is ${\rm Ad}_U(U'')$-equivariant  by Lemma \ref{prop:keyR} (i).

We divide $\Sigma_{>0}$ into the following four disjoint subsets:
\[
\Sigma_{>0}=S_1\sqcup S_2\sqcup S_3\sqcup S_4,
\]
where we set
\begin{align*}
&S_1=\{\epsilon_1-\epsilon_5, \epsilon_2-\epsilon_6, \epsilon_3-\epsilon_7\},\quad 
S_2=\{\epsilon_1-\epsilon_4, \epsilon_2-\epsilon_7, \epsilon_3-\epsilon_6\},\\
&S_3=\{\epsilon_1-\epsilon_7, \epsilon_2-\epsilon_4, \epsilon_3-\epsilon_5\},\quad 
S_4=\{\epsilon_1-\epsilon_6, \epsilon_2-\epsilon_5, \epsilon_3-\epsilon_4\}.
\end{align*}
We also set $\fp_{S_i}=\bigoplus_{\alpha\in S_i}\fp_\alpha$. Notice that each $S_i$ consists of strongly orthogonal roots.
Then, we have an orthogonal decomposition
\begin{align}\label{eq:decG}
\widetilde{\fm}=\bigoplus_{i=1}^4 \widetilde{\fa}_i,
\end{align}
where
\begin{align*}
&\widetilde{\fa}_1:={\rm span}_\R\{\widetilde{V}_7,\widetilde{V}_8\}\subset \fp_{S_1},\quad 
\widetilde{\fa}_2:={\rm span}_\R\{\widetilde{V}_9,\widetilde{V}_{10}\}\subset\fp_{S_2},\\
&\widetilde{\fa}_3:={\rm span}_\R\{\widetilde{V}_{11},\widetilde{V}_{12}\}\subset \fp_{S_3},\quad 
\widetilde{\fa}_4:={\rm span}_\R\{\widetilde{V}_{13},\widetilde{V}_{14}\}\subset \fp_{S_4}.
\end{align*}
Note that each $\widetilde{\fa}_i$ is an  abelian subspace in $\fp$.  Thus, by Lemma \ref{prop:keyR} (i), $\widetilde{\fa}_i$ corresponds to an abelian subspace $\fa_i$ in $\fp'$  contained in $\fm\simeq T_oM$. In particular, each $\widetilde{\fa}_i$ is regarded as a maximal abelian subspace in $\fm$ since the rank of $M=G_2/SO(4)$ is equal to 2.

Now, we are ready to give a proof of Proposition  \ref{prop:keyg}.

\begin{proof}[Proof of Proposition \ref{prop:keyg}]
First, we show that $\widetilde{\Phi}: G_2/SO(4)\to S^{26}(|\eta|)$ is a minimal immersion. Note that, an immersion $\widetilde{F}: M^n\to S^{n+k-1}(r)$ is minimal if and only if the mean curvature of $F:M\to \R^{n+k}$ satisfies
$
H_w=-(n/r^2)\cdot {\bm w}
$
for any $w\in M$, where ${\bm w}$ is the position vector of $F(w)$. 
We show the mean curvature vector of $\Phi: G_2/SO(4)\to \fp$ satisfies this equation. 

Since $\Phi$ is an equivariant isometric immersion, it suffices to consider at the point $\eta=\Phi(o)$. We identify $T_oM$ with $\widetilde{\fm}\subset T_\eta\mathcal{O}_\eta$ as described above. 
We put $g_{ij}:=\langle \widetilde{V}_i,\widetilde{V}_j\rangle$, and denote the inverse matrix of $(g_{ij})_{i,j=7,\ldots, 14}$ by $(g^{ij})_{i,j=7,\ldots, 14}$. Since $\iota: G_{2}/SO(4)\to SO(7)/S(O(3)\times O(4))$ is a totally geodesic immersion, the mean curvature vector $H_{\eta}$ of $\Phi=\varphi\circ \iota: G_{2}/SO(4)\to \fp$ at $\eta$ is expressed by
\[
H_{\eta}=\sum_{i,j=7}^{14}g^{ij}B_{\varphi}(\widetilde{V}_{i},\widetilde{V}_{j}),
\]
where $B_{\varphi}$ is the second fundamental form of the symmetric R-space $\varphi: SO(7)/S(O(3)\times O(4))\to \fp$.
Since \eqref{eq:decG} is an orthogonal decomposition, we actually have
\begin{align}\label{eq:heta}
H_{\eta}=\sum_{i,j=7}^{8}g^{ij}B_{\varphi}(\widetilde{V}_{i},\widetilde{V}_{j})+\sum_{i,j=9}^{10}g^{ij}B_{\varphi}(\widetilde{V}_{i},\widetilde{V}_{j})
+\sum_{i,j=11}^{12}g^{ij}B_{\varphi}(\widetilde{V}_{i},\widetilde{V}_{j})+\sum_{i,j=13}^{14}g^{ij}B_{\varphi}(\widetilde{V}_{i},\widetilde{V}_{j}).
\end{align}
We shall compute the first term. Recall that $\widetilde{\fa}_1$ is a subspace in $T_\eta \mathcal{O}_\eta$ spanned by 
$
\widetilde{V}_7=P_{15}-P_{37}
$
and 
$
\widetilde{V}_8=P_{26}+P_{37}.
$
 Since $P_{ij}$ is a basis of $\fp_{\epsilon_i-\epsilon_j}$ with $|P_{ij}|^2=2C$, we see that
\[
\begin{bmatrix}
g_{77} & g_{78}\\
g_{87} & g_{88}
\end{bmatrix}
=
\begin{bmatrix}
4C & -2C\\
-2C & 4C
\end{bmatrix}
\quad {\rm and}\quad 
\begin{bmatrix}
g^{77} & g^{78}\\
g^{87} & g^{88}
\end{bmatrix}
=
\begin{bmatrix}
\frac{1}{3C} & \frac{1}{6C}\\
 \frac{1}{6C} & \frac{1}{3C}
\end{bmatrix}.
\]
Moreover, since
$
\widetilde{\fa}_1\subset \fp_{S_1}
$
 and $S_1$ consists of strongly orthogonal roots,  Lemma \ref{lem:keyR} and the fact that $\alpha(\eta)=1$ for any $\alpha\in \Sigma_{>0}$ (Lemma \ref{prop:keyR} (ii)) show that 
 \begin{align*}
\sum_{i,j=7}^{8}g^{ij}B_{\varphi}(\widetilde{V}_{i},\widetilde{V}_{j})&=\sum_{i,j=7}^{8}g^{ij}\sum_{\alpha\in S_1}-\langle (\widetilde{V}_i)_\alpha, (\widetilde{V}_j)_\alpha\rangle \alpha^{\sharp}\\
&=-g^{77}\{|P_{15}|^2(\epsilon_1-\epsilon_5)^\sharp+|P_{37}|^2(\epsilon_3-\epsilon_7)^\sharp\}-2g^{78}\{-|P_{37}|^2(\epsilon_3-\epsilon_7)^\sharp\}\\
&\quad -g^{88}\{|P_{26}|^2(\epsilon_2-\epsilon_6)^\sharp+|P_{37}|^2(\epsilon_3-\epsilon_7)^\sharp\}\\
&=-\frac{2}{3}\{(\epsilon_1-\epsilon_5)^\sharp+(\epsilon_2-\epsilon_6)^\sharp+(\epsilon_3-\epsilon_7)^\sharp\}\\
&=-\frac{2}{3}\sum_{\alpha\in S_1}\alpha^\sharp.
\end{align*}
We obtain similar results for the latter three terms in \eqref{eq:heta}. Namely, we have
\begin{align*}
H_{\eta}=\sum_{i=1}^4-\frac{2}{3}\sum_{\alpha\in S_i}\alpha^\sharp=-\frac{2}{3}\sum_{\alpha\in \Sigma_{>0}}\alpha^\sharp=-\frac{2}{3}\Big(4\sum_{i=1}^3\epsilon_i-3\sum_{j=4}^7\epsilon_j\Big)^\sharp.
\end{align*}
On the other hand, by using \eqref{eq:eta} and \eqref{eq:simpr} with $C=14$, we see that the element $\eta$ defined by \eqref{def:eta} is expressed by 
\[
\eta=2\Big(4\sum_{i=1}^3\epsilon_i-3\sum_{j=4}^7\epsilon_j\Big)^\sharp.
\]
 Therefore, we obtain
\[
H_{\eta}=-\frac{1}{3}\eta=-\frac{{\rm dim}M}{|\eta|^2}\cdot \eta,
\]
since ${\rm dim} M={\rm dim} (G_2/SO(4))=8$ and $|\eta|^{2}=24$. This shows that $\widetilde{\Phi}:G_2/SO(4)\to S^{26}(|\eta|)$ is a minimal immersion.

Next, we show that the second fundamental form $B_\Phi$ of $\Phi:G_2/SO(4)\to \fp$ is isotropic. We consider at the point $\eta$ as usual since $\Phi$ is equivariant. Moreover, it is sufficient to show that $|B_\Phi(v,v)|$ is constant for any unit vector $\widetilde{v}\in \widetilde{\fa}_1$. Indeed, for a symmetric space $G/K$ with the associated Cartan decomposition $\fg=\fk\oplus\fm$, it is known that the ${\rm Ad}_G(K)$-action on $\fm$ is a polar action, and a maximal abelian subspace $\fa$ in $\fm$ becomes its section. Namely, any element $X\in \fm$ is expressed by $X={\rm Ad}(k)v$ for some $k\in K$ and $v\in \fa$.  Since $\widetilde{\fa}_1$ corresponds to a maximal abelian subspace $\fa_1$ in $\fm$ and we have an ${\rm Ad}_U(K)$-equivariant isomorphism between $\fm$ and $\widetilde{\fm}$, we see that any element $\widetilde{X}\in \widetilde{\fm}$ is expressed by $\widetilde{X}={\rm Ad}(k)\widetilde{v}$ for some $k\in K$ and $\widetilde{v}\in \widetilde{\fa}_1$. Moreover, because ${\rm Ad}(k)$ is an isometry on $\fp$, we have  
\[
|B_{\Phi}(\widetilde{X},\widetilde{X})|^2=|B_{\Phi}({\rm Ad}(k)\widetilde{v},{\rm Ad}(k)\widetilde{v})|^2=|{\rm Ad}(k)B_\Phi(\widetilde{v},\widetilde{v})|^2=|B_\Phi(\widetilde{v},\widetilde{v})|^2.
\]
Therefore, it is sufficient to consider elements in $\widetilde{\fa}_1$.

 Recall that $\widetilde{\fa}_1$ is a two-dimensional subspace in  $T_\eta\mathcal{O}_\eta$ which is spanned by $\widetilde{V}_7=P_{15}-P_{37}$ and $\widetilde{V}_8=P_{26}+P_{37}$, and the subset $S_1=\{\epsilon_1-\epsilon_5,\epsilon_2-\epsilon_6,\epsilon_3-\epsilon_7\}$ consists of strongly orthogonal roots. Moreover, by using the simple roots, we have
$
 \epsilon_{1}-\epsilon_{5}=\sum_{i=1}^{4}(\epsilon_{i}-\epsilon_{i+1}),
$
and hence, the relation \eqref{eq:simpr} shows that
\[
 |\epsilon_1-\epsilon_5|^{2}=\Big|\sum_{i=1}^{4}\frac{1}{C}H_{i}\Big|^{2}=\frac{1}{C^{2}}\sum_{i,j=1}^{4}\langle H_{i},H_{j}\rangle=\frac{2}{C},
\]
where we used the facts $|H_{i}|^{2}=2C$ and $\langle H_{i}, H_{i+1}\rangle=\langle H_{i-1}, H_{i}\rangle=-C$ (otherwise $\langle H_{i}, H_{j}\rangle=0$).
Similar computations show that
  \[
 |\epsilon_1-\epsilon_5|= |\epsilon_2-\epsilon_6|= |\epsilon_3-\epsilon_7|=\sqrt{\frac{2}{C}}. 
\]
Therefore, we can apply Lemma \ref{lem:keyR1} to obtain 
\begin{align}\label{eq:bsq}
|B_\Phi(v,v)|^2=|B_\varphi(v,v)|^2=\frac{1}{2}\cdot \frac{2}{C}=\frac{1}{C}
\end{align}
for any unit tangent vector $v\in \widetilde{\fa}_1$, where we used the fact that $\iota: G_2/SO(4)\to G_3(\R^7)$ is totally geodesic. This proves that $B_{\Phi}$ is isotropic.
\end{proof}

\subsection{Proof of Theorem \ref{thm:G}}
Recall that when $\widetilde{F}: M^{n}\to S^{n+k-1}(r)$ is a minimal immersion, Corollary \ref{cor:cri1} shows that $M$ has the property $({\bf N}_1)$ if $F: M^{n}\to \R^{n+k}$  satisfies the following inequality (see \eqref{eq:q2}):
\begin{align}\label{eq:des}
{\rm Ric}>\frac{1}{2}\Big(\frac{n}{r^{2}}-\beta_{F}\Big)g.
\end{align}

By Proposition \ref{prop:keyg} (i),  the symmetric space $M=G_{2}/SO(4)$ is minimally immersed into $S^{26}(|\eta|)$ via the map $\Phi: G_{2}/SO(4)\to \fp$. We shall show $F=\Phi$ satisfies the inequality \eqref{eq:des}.

\begin{proof}[Proof of Theorem \ref{thm:G}]

 We denote the induced metric on $M=G_2/SO(4)$ via the immersion $\Phi$ by $g$. Due to the construction of $\Phi$, it holds that  $g=c\cdot g_{K}$ for some positive constant $c$, where $g_K$ is the standard metric on the symmetric space $M$ defined by the Killing-form of $G_2$. We shall determine the constant $c$. 

First, we derive an explicit formula of the Killing-form on $\fg_2$.
We denote  the complexification of $\fg_{2}$ and $\fo(7)$ by $\fg_{2}^{\C}$ and $\fo(7,\C)$, respectively.  We write the Killing forms on  $\fo(7,\C)$ and $\fg_{2}^{\C}$ by ${\bf B}_{1}$ and ${\bf B}_{2}$, respectively. It is known that ${\bf B}_{1}(X,Y)=5{\rm tr}(XY)$ for $X, Y\in \fo(7,\C)$. 

As mentioned above, $\fg_{2}$ is realized as a subalgebra of $\fo(7)$. Thus, we regard $\fg_{2}^{\C}$ as a Lie subalgebra of $\fo(7,\C)$ and  set ${\bf B}_{2}':={\bf B}_{1}|_{\fg_{2}^{\C}}$. Since ${\bf B}_{1}$ is a non-degenerate symmetric $2$-form satisfying that ${\bf B}([X,Y],Z)+{\bf B}(Y,[X,Z])=0$, so does ${\bf B}_{2}'$ on $\fg_{2}^{\C}$. Thus,  Schur's lemma implies that ${\bf B}_{2}'$ coincides with a constant multiple of ${\bf B}_{2}$. In particular, we may write ${\bf B}_{2}(X,Y)=\lambda {\rm tr}(XY)$ for some constant $\lambda$, where $X,Y\in \fg_{2}^{\C}$ and we regard $X, Y$ as elements in $\fo(7,\C)$. By a direct computation by using the basis $\{V_{1},\ldots, V_{14}\}$ of $\fg_{2}$, we see that ${\bf B}_{2}(V_{1},V_{1}):={\rm tr}({\rm ad}(V_{1})^{2})=-16$, where ${\rm ad}(V_{1})^{2}: \fg_{2}\to \fg_{2}$ is an endomorphism on $\fg_{2}$. On the one hand, we have 
${\rm tr}(V_{1}^{2})=-4$ as $V_{1}\in \fo(7,\C)$.  Therefore $\lambda=(-16)/(-4)=4$, and hence, we obtain ${\bf B}_{2}(X,Y)=4{\rm tr}(XY)$ for any $X, Y\in \fg_{2}^\C\subset \fo(7,\C)$.  Moreover, the restriction ${\bf B}_2|_{\fg_2}$ yields the Killing-form on $\fg_2$.
  
Then, the inner product on $\fm\subset \fg=\fg_{2}$ defined by the Killing form ${\bf B}_{2}$ is expressed  by $\langle X, Y\rangle_{K}:=-{\bf B}_{2}(X,Y)=-4{\rm tr}(XY)$. By letting $X=Y=V_{7}=G_{15}-G_{37}\in \fg_{2}$, we see $|V_{7}|_{K}^{2}=16$. On the other hand, $V_{7}$ corresponds to $\widetilde{V}_{7}=P_{15}-P_{37}$ in $\widetilde{\fm}\subset T_{\eta}\mathcal{O}_{\eta}\subset \fp$, and the square norm with respect to the inner product on $\fp$ is given by $|\widetilde{V}_{7}|^{2}_{\fp}=4C$. Therefore, the constant $c$ is given by $c=(4C)/16=7/2$.

Since $g=c\cdot g_K$, the Ricci curvature ${\rm Ric}$ w.r.t. $g$ is the same as ${\rm Ric}_{K}$ w.r.t. $g_{K}$. Recall that we have ${\rm Ric}_{K}=\frac{1}{2}g_{K}$ if $M=G/K$ is a semi-simple compact Riemannian symmetric space. In particular, we have  ${\rm Ric}={\rm Ric}_{K}=\frac{1}{2}g_{K}=\frac{1}{2c} g=\frac{1}{7}g$. On the other hand, by Proposition \ref{prop:keyg}, $B_{\Phi}$ is isotropic and  we know that $\beta_{\Phi}=1/C=1/14$ by \eqref{eq:bsq}. Therefore, we see
\[
\frac{1}{2}\Big(\frac{n}{r^{2}}-\beta_{\Phi}\Big)=\frac{1}{2}\Big(\frac{8}{24}-\frac{1}{14}\Big)=\frac{11}{84}<\frac{12}{84}=\frac{1}{7}.
\]
Namely, the inequality \eqref{eq:des} is satisfied. Thus, $M=G_{2}/SO(4)$ has  property $({\bf N}_1)$ .
\end{proof}

\subsection{Remark on Ohnita's conjecture}

We remark that the above computational result also confirms Ohnita's conjecture given in \cite{Ohnita3}  for the case when $M=G_2/SO(4)$. 

In \cite{Hel2}, Helgason proved that, for any irreducible compact Riemannian symmetric space $M$, there exists a totally geodesic submanifold in $M$ whose sectional curvature is constant and the constant coincides with the maximum of the sectional curvature of $M$.  Moreover, whenever $M\neq \R P^n$, such a totally geodesic submanifold with maximal dimension is topologically sphere and is called the {\it Helgason sphere}. In \cite{Ohnita3}, Ohnita showed that the Helgason sphere is actually a stable minimal submanifold in $M$. Moreover, he conjectured that, if $M$ is a simply-connected irreducible compact Riemannian symmetric space, then there exists no stable rectifiable $p$-current in $M$ for $1\leq p\leq m-1$ or $n-m+1\leq p\leq n-1$, where $n={\rm dim}M$ and $m$ is the dimension of the Helgason sphere. Although a partial answer was given by himself \cite[Theorem 7.7]{Ohnita3},  this conjecture is still open.

When $M=G_{2}/SO(4)$,  the dimension of Helgason sphere is equal to $2$ (see \cite{Ohnita3}), and hence, there is a $2$-dimensional stable minimal submanifold in $G_{2}/SO(4)$. However, we obtain the following result, which confirms Ohnita's conjecture for $M=G_{2}/SO(4)$.

\begin{proposition}
There is no stable rectifiable 1-current (or $7$-current) in the compact symmetric space $G_{2}/SO(4)$. 
\end{proposition}
\begin{proof}
Apply \cite[Theorem 7.4]{Ohnita3} to the minimal immersion $\widetilde{\Phi}:G_2/SO(4)\to S^{26}(|\eta|)$. Note that we have proved that $0<2{\rm Ric}-(n/r^2-\beta_{\Phi})g=2{\rm Ric}-((n-1)/r^2-\beta_{\widetilde{\Phi}})g$.
\end{proof}

\subsection*{Acknowledgements}
The author would like to thank Pak Tung Ho, Yoshihiro Ohnita and Ryokichi Tanaka for their interest and for helpful discussion. He also thanks to Yuuki Sasaki for letting me know about immersions of $G_2/SO(4)$. The author is supported by JSPS KAKENHI Grant Number JP18K13420 and 23K03122.

\end{document}